\documentclass[11pt]{article}
\usepackage[T1]{fontenc}
\usepackage[latin9]{inputenc}
\usepackage{mathtools}
\usepackage{amssymb}
\usepackage{stackrel}
\usepackage{hyperref}
\hypersetup{colorlinks=true, linkcolor=blue, filecolor=magenta, urlcolor=blue}
\usepackage{amsthm,cleveref}
\usepackage[dvipsnames]{xcolor}
\usepackage{mathrsfs}
\usepackage[margin = 3 cm,marginpar= 2cm]{geometry}
\usepackage{amsthm,cleveref}

\makeatletter
\newtheorem{thm}{Theorem}[section]
\newtheorem{defn}[thm]{Definition}
\newtheorem{prop}[thm]{Proposition}
\newtheorem{lem}[thm]{Lemma}
\newtheorem{cor}[thm]{Corollary}

\newtheorem{rem}[thm]{Remark}

\newtheorem*{assumptions}{Assumption}
\numberwithin{equation}{section}

\@ifundefined{date}{}{\date{}}
\makeatother

\usepackage{babel}
\usepackage{natbib}

\newcommand{\R}{\mathbb{R}}

\def\dd{{\rm d}}
\def \de{{\rm d}}

\begin{document}
\title{Well-posedness and mean-field limit of discontinuous weighted dynamics via the relative entropy method}
\author{
Immanuel Ben-Porat%
     \thanks{Universit\"at Basel
 Spiegelgasse 1
 CH-4051 Basel, Switzerland.  {Email: immanuel.ben-porath@unibas.ch}} %
\and
    Jos\'e A. Carrillo%
     \thanks{Mathematical Institute, University of Oxford, Oxford OX2 6GG, UK.  {Email: carrillo@maths.ox.ac.uk}} % 
     \and
    Alexandra Holzinger%
     \thanks{Mathematical Institute, Utrecht University, Budapestlaan 6 , Hans Freudenthalgebouw, 3584 CD Utrecht. {Email: a.holzinger@uu.nl}} %
}

\maketitle
\begin{abstract}
We consider deterministic particle dynamics with time evolving weights and their associated Kolmogorov equation and mean-field equation. We prove existence and unique\-ness for the limit PDE alongside estimates on the growth of the logarithmic gradient as well as existence of weak solutions for the Kolmogorov equation satisfying an appropriate entropy inequality. We then apply these estimates and the relative entropy method as developed in \cite{jabin2018quantitative}, in order to derive the associated  equation as a mean field limit. Our results cover both interactions and influence kernels with  mild regularity assumptions.         
\end{abstract}
\section{Introduction }
In this paper we are concerned with the mean-field limit
of the following system of ODEs
\begin{equation}
\left\{ \begin{array}{lc}
\dot{x_{i}}^{N}(t)=\frac{1}{N}\stackrel[j=1]{N}{\sum}m_{j}^{N}(t)\mathbf{a}(x_{j}^{N}(t)-x_{i}^{N}(t)),\ x_{i}^{N}(0)=x_{i}^{0,N}\\
\dot{m}_{i}^{N}(t)=\frac{1}{N}\stackrel[j=1]{N}{\sum}S(x_{i}^{N}(t),m_{i}^{N}(t),x_{j}^{N}(t),m_{j}^{N}(t)),\ m_{i}^{N}(0)=m_{i}^{0,N}.
\end{array}\right.\label{particles}
\end{equation}
These kind of systems find applications in opinion dynamics \cite{ayi2023graph,ayi2024large,gkogkas2025mean}, neuroscience \cite{duchet2023mean} and other scientific fields, see \cite{berner2023adaptive} and the references therein. 

Here  $\mathbf{a}:\mathbb{T}^{d}\rightarrow\mathbb{R}^{d}$ is a given
interaction kernel and $S:(\mathbb{T}^{d}\times\mathbb{R}_{+})^{2}\rightarrow\mathbb{R}$ is a given
influence kernel. We are interested in the large population limit of \eqref{particles} as $N\to \infty$ of the distribution $\psi_N$ of $(x_i,m_i)_{i=1}^N $, which fulfils the \textit{Kolmogorov equation}
\begin{align} \label{Kolmogorov eq} 
\begin{cases}
\partial_{t}\psi_{N}+\frac{1}{N}\stackrel[i=1]{N}{\sum}\textnormal{div}_{x_{i}}\bigg(\psi_{N}\stackrel[j=1]{N}{\sum}m_{j}\mathbf{a}(x_{i}-x_{j})\bigg)
+\frac{1}{N}\stackrel[i=1]{N}{\sum}\partial_{m_{i}}\bigg(\psi_{N}\stackrel[j=1]{N}{\sum}S(x_{i},m_{i},x_{j},m_{j})\bigg)=0\\
\psi_{N}(0,\mathbf{x}_{N},\mathbf{m}_{N})=\psi_{N,0}.
\end{cases}
\end{align} Here, the unknown $\psi_{N}(t,\mathbf{x}_{N},\mathbf{m}_{N})$
is a time dependent probability density 
$\psi_{N}(t,\cdot)\in\mathscr{P}(\mathbb{T}^{dN}\times\mathbb{R}_{+}^{N})$ and we apply the shorthand notation $(\mathbf{x}_{N},\mathbf{m}_{N})\coloneqq(x_{1},\ldots,x_{N},m_{1},\ldots,m_{N})$. The reason of viewing the ODE system in $(x,m)$ lies in the fact that the Kolmogorov equation for the spatial distribution of $\mathbf{x}_N$ in \eqref{particles} is not a closed equation. 

The ultimate goal of this article is to show a propagation of chaos result for $\psi_N$ in the large population limit $N \to \infty$ towards the product measure $\psi^{\otimes N}$, where $\psi$ is a solution of the following non-local PDE on $\mathbb{T}^d \times \R_+$
\begin{align}\label{limit PDE Intro}
\begin{cases}
\partial_{t} \psi + \textnormal{div}_x\big( \psi \,\mathbf{a}\star_x \mu[\psi]\big) + \partial_m (\psi \,\mathbf{S} [\psi]) = 0, \\
 \psi(0,x,m)=\psi_{0}.   
\end{cases}
\end{align}
Here, $\mu[\psi]$ is defined as
\begin{equation}\label{mudef}
    \mu[\psi](t,x):= \int_{\R_{+}}n\psi(t,x,n) \ \dd n\,,
\end{equation} 
designating the first moment with respect to the weight variable and
\begin{align*}
\mathbf{S} [\psi](x,m) &:= \int_{\mathbb{T}^d \times \mathbb{R}_{+}}  S(x,m,y,n) \psi(y,n) \ \dd n \dd y. 
\end{align*}
The solution to  \eqref{limit PDE Intro} is a time-dependent probability density $\psi(t,\cdot)\in \mathscr{P}(\mathbb{T}^{d}\times \mathbb{R}_{+}
)$. The propagation of chaos property of \eqref{particles} is understood in the following $L^1$-sense: For $1\leq k \leq N$ set
\begin{align}
\psi_{N:k}(\mathbf{x}_{k},\mathbf{m}_{k})\coloneqq \int_{\mathbb{T}^{d(N-k)}\times \mathbb{R}_{+}^{N-k}}\psi_{N}(\mathbf{x}_{k},\mathbf{x}_{N-k},\mathbf{m}_{k},\mathbf{m}_{N-k})\ \de\mathbf{x}_{N-k}\de\mathbf{m}_{N-k}. \label{k marginal}   
\end{align}
The density $\psi_{N:k}\in \mathscr{P}(\mathbb{T}^{dk}\times \mathbb{R}_{+}^{k})$ is called the $k$-th marginal of $\psi_{N}$. Typically, one aims to show that for any $k \in \mathbb{N}$, $\psi_{N:k}$ converges in $L^1((\mathbb{T}^{d}\times \R_{+})^{k})$ to the product measure $\psi^{\otimes k}$, i.e. 
\begin{equation}\label{aux}
\underset{t\in [0,T]}{\sup}\left\Vert \psi_{N:k}(t,\cdot)-\psi^{\otimes k}(t,\cdot)\right\Vert _{L^{1}((\mathbb{T}^{d}\times \mathbb{R}_{+})^{k})}\underset{N\rightarrow \infty}{\rightarrow} 0
\end{equation} 
on a suitable time interval $[0,T]$. 
This convergence also implies the more classical probabilistic propagation of chaos, that is,
the weak convergence of the empirical measure
$\mu_{N}(t):=\frac{1}{N}\sum_{i=1}^N\delta_{(x,m)=(x_{i}(t),m_{i}(t))}$ towards $\psi$.
The underlying strategy of our work is to adapt the well-known relative entropy approach for quantitative estimates of the mean-field limit as in \cite{jabin2018quantitative} to the present settings. We aim at analyzing the distance between $\psi_{N:k}$ and $\psi^{\otimes k}$ by showing the vanishing as $N\rightarrow \infty$ of the relative entropy, i.e.
\begin{align}\label{defn_of_entropy}
\mathcal{H}_{N}(t)&\coloneqq\mathcal{H}(\psi_{N}(t,\cdot)| \psi^{\otimes N}(t,\cdot))\\
&= \frac{1}{N}\int_{\mathbb{T}^{dN}\times\R_{+}^{N}} \psi_{N}(t,\mathbf{x}_{N},\mathbf{m}_{N}) \log\bigg(\frac{\psi_{N}(t,\mathbf{x}_{N},\mathbf{m}_{N})}{\psi^{\otimes N}(t,\mathbf{x}_{N},\mathbf{m}_{N})}\bigg)\ \dd \mathbf{x}_{N}\de \mathbf{m}_N \underset{N\rightarrow \infty}{\rightarrow} 0 ,\notag
\end{align}
uniformly in $t$ on some short time interval. It is classical to deduce \eqref{aux} and the mean-field limit from \eqref{defn_of_entropy} via the Csisz\'ar-Kullback-Pinsker (aka CKP) inequality, which provides an estimate of the form 
\begin{align*}
\left\Vert \psi_{N:k}-\psi^{\otimes k}\right\Vert_{1}\leq \sqrt{2k\mathcal{H}(\psi_{N}\mid \psi^{\otimes N})}. \end{align*}

\subsection{Previous results }

The study of mean-field limits with time evolving
weights has recently drawn the attention of the mathematical community.
Let us first remark that if $S$ takes the form $S(x,m,y,n)=mns(x,y)$ for some given $s:\mathbb{T}^{2d}\rightarrow \mathbb{R}$, then, by taking the first moment of \eqref{limit PDE Intro}
with respect to the weight variable $m$, one readily checks that $\mathcal{\mu}(t,x)$ (as defined in \eqref{mudef})
is governed by the non-local transport equation with source term 

\begin{equation}
\partial_{t}\mu+\mathrm{div}_{x}(\mu\,\mathbf{a}\ast\mu)=h\left[\mu\right]\label{eq:-53}
\end{equation}
where the source term $h[\mu]$ is given by 

\[
h\left[\mu\right](x)\coloneqq\mu(x)\int_{\mathbb{T}^{d}}s(x,y)\mu(y)\ \dd y. 
\]
In accordance, one considers the \textit{weighted} empirical measure
$\widetilde{\mu}_{N}(t)\coloneqq\frac{1}{N}\stackrel[i=1]{N}{\sum}m_{i}(t)\delta_{x=x_{i}(t)}$.
Assuming that $\mathbf{a}$ and $s$ are Lipschitz and that $s$ satisfies an appropriate parity condition, in \cite{pouradier2021mean}
the mean-field limit equation \eqref{eq:-53} on the whole space has been derived directly from the trajectories \eqref{particles} by proving $W_{1}(\widetilde{\mu}_{N}(t,\cdot),\mu(t,\cdot))\underset{N\rightarrow\infty}{\rightarrow}0$
uniformly in time provided this is true initially, i.e. provided $W_{1}(\widetilde{\mu}_{N}(0,\cdot),\mu(0,\cdot))\underset{N\rightarrow\infty}{\rightarrow}0$. In the same work
the well posedness of the Cauchy problem associated with \eqref{eq:-53} has been
also addressed. The well-posedness theory of equations of
the type \eqref{eq:-53}, or variants thereof, has been also handled in \cite{piccoli2014generalized,piccoli2019wasserstein}. The mean field limit and well posedness for the 1D
whole space case where $\mathbf{a}$ exhibits an attractive Coulombic singularity
$\mathbf{a}(x)=\mathrm{sgn}(x)$ has been handled in \cite{porat2023mean} via
the theory of entropy solutions for 1D conservations laws. Recently,
the same problem  where $\mathbf{a}$ exhibits a repulsive Coulombic singularity in higher
dimensions has been studied in \cite{ben2025singular} using a commutator estimate
approach. The case where $\mathbf{a}$ is Lipschitz but $s$
has a jump discontinuity at the origin has been introduced in \cite{mcquade2019social}
where it is referred to as the pairwise competition model, alongside
the study of long time behavior. The weak mean field limit of the
pairwise competition model has been established in \cite{ben2024graph} as a
consequence of the graph limit of the underlying dynamics. The major
limitation of this approach is that it necessitates choosing the initial
datum (both at the microscopic and macroscopic level) in a specific
manner -- in particular, $L^{p}$ initial data is not admissible with
this approach. Other works which consider the graph limit and its link with the mean field limit include 
\cite{ayi2023graph,biccari2019dynamics, paul2022microscopic}.  The case where the weights form a non-symmetric matrix is also interesting, as it renders the system to be non-exchangeable, and has been considered in \cite{ ayi2026mean, gkogkas2025mean, jabin2025mean}. An extensive overview of opinion dynamic models and their various exchangeable and non-exchangeable limits can be found in \cite{ayi2024large, paul2022microscopic}.\\
A different active area in the theory of mean field limits is the relative entropy method, which has been initiated in the work \cite{jabin2016mean}, and further developed within \cite{, bresch2019modulated, bresch2023mean, jabin2018quantitative}. Let us briefly recall  the main idea of the relative entropy method for the simpler case where $S=0$ and $m_{j}=1$ for all $1\leq j\leq N$. The associated Kolmogorov equation in this case is also known as the Liouville equation and writes
\begin{align}\label{liouville_S0}  
\partial_{t}\psi_{N}+\frac{1}{N}\stackrel[i=1]{N}{\sum}\textnormal{div}_{x_{i}}\bigg(\stackrel[j=1]{N}{\sum}\mathbf{a}(x_{i}-x_{j})\psi_{N}\bigg)
=0, \ \psi_{N}(0,\mathbf{x}_{N})=\psi_{N,0} \end{align}
while the limit equation is the homogeneous non-local transport equation 
\begin{align} \label{non local transport}  
\partial_{t}\psi+\mathrm{div}_{x}(\psi\,\mathbf{a}\star \psi)=0, \ \psi(0,x)=\psi_{0}.  
\end{align}
The quantity whose evolution is studied is the \textit{relative entropy}, defined as follows. Given $\psi\in\mathscr{P}(\mathbb{T}^{d})$
we define $\overline{\psi_{N}}\coloneqq\psi^{\otimes N}\in \mathscr{P}(\mathbb{T}^{dN})$ and for
each $\psi_{N}\in\mathscr{P}(\mathbb{T}^{dN})$
we define the relative entropy analogously to the case of varying weights by 
\begin{align*}
\mathcal{H}(\psi_{N}\left|\overline{\psi_{N}}\right.)\coloneqq\frac{1}{N}\int_{\mathbb{T}^{dN}}\psi_{N}(\mathbf{x}_{N})\log\left(\frac{\psi_{N}(\mathbf{x}_{N})}{\overline{\psi_{N}}(\mathbf{x}_{N})}\right)\ \de \mathbf{x}_{N}.  
\end{align*}
A remarkable observation made in \cite{jabin2016mean} is that if $\mathbf{a}$ is bounded and divergence free ($\mathrm{div}(\mathbf{a})=0$), the time dependent relative entropy 
$\mathcal{H}_{N}(t)\coloneqq\mathcal{H}(\psi_{N}(t,\cdot)\left|\overline{\psi_{N}}(t,\cdot)\right.)$, where $\psi(t,\cdot)$ is taken to be the solution of \eqref{non local transport} and $\psi_{N}(t,\cdot)$
the solution of \eqref{liouville_S0}, satisfies the evolution estimate 
\begin{align}\label{entropy_estimate_classical}
    \mathcal{H}_{N}(t)\leq\left(\mathcal{H}_{N}(0)+O\left(\frac{1}{N}\right)\right)e^{t}.
\end{align}
Depending on the specifics of the problem, the above estimate is valid either on any time interval on which the solution exists or on a time interval possibly shorter than the maximal existence time.  
\subsection{Main results and assumptions}
In this work we are concerned with the mean field limit leading from \eqref{Kolmogorov eq} to \eqref{limit PDE Intro} in the case where both $\mathbf{a}$ and $S$ exhibit some roughness. More specifically, $\mathbf{a}$ is merely bounded and divergence free, which is essentially the same class of kernels considered in \cite{jabin2016mean}, and $S$ exhibits $W^{1,1}$ regularity in the variables $x,y$. However, the utility of the $W^{1,1}$ regularity of $S$ is used in the proof that $\psi_{N}$ satisfies an appropriate entropy inequality, and is not needed for what concerns the existence and uniqueness of the limit PDE. The relative entropy approach allows us to derive the mean-field limit
of the system \eqref{particles} for the case where $\mathbf{a}$ and $S$ exhibit very mild regularity, which clearly extends \cite{pouradier2021mean}. 
We start by stating the assumptions on the interaction kernel $\textbf{a}$, influence kernel $S$ and the initial data:
\begin{assumptions}[\textbf{H1}] 
    \label{assumptions kernel}  We impose the following assumptions on the kernels in \eqref{Kolmogorov eq}.
\begin{itemize}
\item[i.] \textbf{Regularity of interaction kernel: }$\mathbf{a\in}L^{\infty}(\mathbb{T}^{d};\mathbb{R}^{d})$ and $\mathrm{div}_{x}(\mathbf{a})=0$. 
\item[ii.] \textbf{Regularity of $S$: } $S\in L^{\infty}_{x,y,n}W^{2,\infty}_{m}$ and there exist some $\overline{S}_{0},\overline{S}_{1},\overline{S}_{2}>0$ such that for a.e. $(x,m,y,n)\in \mathbb{T}^
{d}\times\mathbb{R}_{+}\times \mathbb{T}^{d}\times \mathbb{R}_{+}$ one has 
\begin{align*}
\left\vert S(x,m,y,n)\right\vert\leq \overline{S}_{0}, \ \left\vert \partial_{m}S(x,m,y,n)\right\vert\leq \overline{S}_{1},\ \left\vert \partial_{mm}S(x,m,y,n)\right\vert\leq \overline{S}_{2}  .      
\end{align*}
In addition, for any $(x,y,n)\in \mathbb{T}^{d}\times \mathbb{T}^{d}\times \mathbb{R}_{+}$ it holds that $\mathrm{supp}(S(x,\cdot,y,n))\Subset (0,\infty)$. 
Furthermore, we assume that for all $\psi\in W^{1,1}(\Omega)$ and $\alpha\in \{0,1\}$ we have 
\begin{align}\label{definition_gradient_xy}
\left\vert \nabla_{x}\int_\Omega\partial_{m}^{\alpha} S(x,m,y,n) \psi(y,n) \ \de y\de n \right\vert=\left\vert\int_\Omega\partial_{m}^{\alpha} S(x,m,y,n) \nabla_{y}\psi(y,n)\  \de y\de n\right\vert.  
\end{align}
\end{itemize} 
\end{assumptions}

\begin{rem}
Condition \eqref{definition_gradient_xy} is implied by symmetry or anti-symmetry of the kernel $S$ in $(x,y)$. As an example of a kernel $S$ verifying \textbf{H1}ii. we may take $S(x,m,y,n)=\chi_{1}(m)\chi_{2}(n)s(x-y)$ where $\chi_{1},\chi_{2}$ 
are smooth compactly  supported and $s\in L^{\infty}(\mathbb{T}^{d})$ an even or odd function.  
\end{rem}
\begin{assumptions}[\textbf{H2}] We require the following regularity of the initial condition of the limiting equation \eqref{limit PDE Intro}. 
\begin{itemize}
\item[i.] \textbf{Regularity:} $\psi_{0}\in C^{\infty}(\mathbb{T}^{d}\times \mathbb{R}_{+})\cap W^{1,1}(\mathbb{T}^{d}\times \mathbb{R}_{+})\cap W^{1,\infty}(\mathbb{T}^{d}\times \mathbb{R}_{+})\cap\mathscr{P}(\mathbb{T}^{d}\times \mathbb{R}_{+})$. 
\item[ii.] \textbf{Initial positivity, bound of logarithmic gradient and moments:} There exists a constant $\mathcal{C}>1$ such that for all $(x,m)\in \Omega$ it holds that 
\begin{align*}
\left\vert\partial_{m}\psi_{0}(x,m)\right\vert \leq \mathcal{\mathcal{C}}e^{-m}&, \qquad \left\vert \nabla_{x}\psi_{0}(x,m)\right\vert\leq \mathcal{C}e^{-m}  
\end{align*}
and 
\begin{align*}
\frac{1}{\mathcal{C}}e^{-m}\leq  \psi_{0}(x,m) \leq \mathcal{C}e^{-m}.     
\end{align*}
In addition there exists $n_0>0$ such that
\begin{align*}
\underset{b\geq 1}{\sup}\frac{\mathfrak{M}_{b,1}^{\frac{1}{b}}}{b}\leq n_{0}
\end{align*}
with the notation $\mathfrak{M}_{b,p}\coloneqq \int_{\mathbb{T}^{d}\times \mathbb{R}_{+}}(1+m^{b p})\psi_{0}^{p}(x,m)\ \de x \de m$.   
\end{itemize}
\end{assumptions}
\begin{rem}
\label{initial data gamma fun}
We demonstrate an initial data realizing (i) -- (ii). Take any  $f\in C^{\infty}(\mathbb{T}^{d})\cap W^{1,1}(\mathbb{T}^{d})\cap W^{1,\infty}(\mathbb{T}^{d})\cap \mathscr{P}(\mathbb{T}^{d})$ with $f>0$ and let $\psi_{0}=e^{-m}f(x)$. To see that (ii) is satisfied note that $\psi_{0}\in \mathscr{P}(\mathbb{T}^{d}\times \mathbb{R}_{+})$ as a tensor product of probability densities. In addition, $\left\Vert \partial_{m}\psi_{0}\right\Vert_{\infty}\leq \left\Vert f\right\Vert_{\infty}e^{-m}$ and $\left\Vert \nabla_{x}\psi_{0}\right\Vert_{\infty}\leq \left\Vert \nabla f\right\Vert_{\infty}e^{-m}$.  Furthermore, notice that 
\begin{align*}
\int_{0}^{\infty}m^{b}\psi_{0}(x,m)\ \de x\de m= \int_{0}^{\infty}m^{b}e^{-m}\ \de m=\Gamma(b+1)=b!.     
\end{align*}
Note also that point ii. in particular implies that  $\mu[\psi_{0}]\in L^{1}(\mathbb{T}^{d})\cap L^{\infty}(\mathbb{T}^{d})$ with the bound $\left\Vert \mu[\psi_{0}]\right\Vert_{L^{\infty}\cap L^{1}}\leq \overline{\mu}$ for some $\overline{\mu}>0$, an observation which will be used in the sequel.  
\end{rem}

\begin{assumptions}[\textbf{H3}] Concerning the initial data of \eqref{Kolmogorov eq}, we fix $N$ and assume  $\psi_{N,0}\in \mathscr{P}(\mathbb{T}^{dN}\times \mathbb{R}_{+}^{N})\cap L^{1}(\mathbb{T}^{dN}\times \mathbb{R}_{+}^{N})$  fulfills the following assumptions. 
\begin{itemize}
    \item[i.] \textbf{Exponential decay:} There exist some $\overline{\mathbf{E}}_{N}>0$ such that 
\begin{align}
\left\Vert  e^{\frac{1}{2}\sum_{k=1}^{N}m_{k}}\psi_{N,0}(\mathbf{x}_{N},\mathbf{m}_{N})\ \right\Vert_{\infty}  =:\overline{\mathbf{E}}_{N} < \infty \label{initial exp moments}.   \end{align}

\item[ii.] \textbf{Bounded Entropy:} There is some constant $\overline{\mathbf{H}}_N>0$ such that  
\begin{align}
\int_{\mathbb{T}^{dN}\times \mathbb{R}_{+}^{N}}\psi_{N,0}(\mathbf{x_{N}},\mathbf{m}_{N}) \log (\psi_{N,0}(\mathbf{x_{N}},\mathbf{m}_{N})) \ \dd \mathbf{x}_{N}\dd \mathbf{m}_{N}=:\overline{\mathbf{H}}_N < \infty.    
\end{align}

\item[iii.] \textbf{Propagation of chaos at $t=0$:}   \begin{align}
        \mathcal{H}_N(0)= \mathcal{H}(\psi_{N,0}| \psi_0^{\otimes N}) \to 0 ~~~\textnormal{as }N \to \infty. \label{initial vanishing of entropy}
    \end{align}
\end{itemize}  
\end{assumptions}
Assumptions \eqref{initial exp moments}-\eqref{initial vanishing of entropy} are satisfied if we choose for instance $\psi_{N,0}=\psi_{0}^{\otimes N}$, where $\psi_{0}$ satisfies (\textbf{H2}). Our first theorem ensures that the limiting PDE \eqref{limit PDE Intro} is locally well-posed. When $\mathbf{a}$ and $S$ are Lipschitz global well-posedness is endured, but finite time blow-up may occur when $\mathbf{a}$ is merely bounded or even $W^{1,p}$ for $1\leq p<\infty$ if the initial data has low regularity, see \cite{carrillo2011global}. Thus local well-posedness is natural.  
\begin{defn}[Strong solution to \eqref{limit PDE Intro}]
A function $\psi\in C([0,T];L^{1}\cap L^{\infty}(\mathbb{T}^{d}\times \mathbb{R}_{+}))\cap L^{\infty}([0,T];W^{1,\infty}(\mathbb{T}^{d}\times \mathbb{R}_{+})\cap W^{1,1}(\mathbb{T}^{d}\times \mathbb{R}_{+}))$ with $\psi(t,\cdot)\in\mathscr{P}(\mathbb{T}^{d}\times \mathbb{R}_{+})$ for all $t\in [0,T]$ is a strong solution to \eqref{limit PDE Intro} if $\mu[\psi]\in L^{\infty}([0,T];W^{1,1}(\Omega))$ and $\psi$ satisfies \eqref{limit PDE Intro} for a.e. $t>0$, $x\in\mathbb{T}^{d}$ and $m\in\mathbb{R}_{+}$.
\end{defn}

\begin{thm}[Limiting PDE \eqref{limit PDE Intro}] \label{well posedness intro}
    Let assumptions \textbf{H1}-\textbf{H2} hold.   
\begin{enumerate}
        \item \textbf{Existence:} There exists some $T_{\ast}>0$ and a strong solution $\psi(t,\cdot)$ to \eqref{limit PDE Intro} on $ [0,T_{\ast}]$.
    \item \textbf{Uniqueness:} Let $\psi_{1},\psi_{2}$ be strong solutions to \eqref{limit PDE Intro}
on $[0,T]$ with initial data $\psi_{0}^{1},\psi_{0}^{2}$. 
Then, there is some constant $C>0$ such that 
\begin{align*}
\left\Vert (\psi_{1}-\psi
_{2})(t,\cdot)\right\Vert_{1}+\left\Vert (\mu[\psi_{1}]-\mu[\psi_{2}])(t,\cdot)\right\Vert_{1}\leq e^{Ct}\left(\left\Vert \psi_{0}^{1}-\psi_{0}^{2}\right\Vert_{1}+\left\Vert \mu[\psi_{0}^{1}]-\mu[\psi_{0}^{2}]\right\Vert_{1} \right)       
\end{align*}
for all $t\in [0,T]$. Consequently, uniqueness of strong solutions of \eqref{limit PDE Intro} follows. 

    \item \textbf{Propagation  of logarithmic gradient bounds:} Let $\psi$  be the unique solution of \eqref{limit PDE Intro} ensured by 1. and 2.
    Then, there exist some $C>0$ such that for a.e. $(t,x,m)\in[0,T_{\ast}]\times \mathbb{T}^{d}\times \mathbb{R}_{+}$ it holds that     
    \[
    \left|\nabla_{x}\log(\psi)(t,x,m)\right|\leq C
    \qquad
    \mbox{and}     
    \qquad
    \left|\partial_{m}\log(\psi)\right|(t,x,m)\leq C.
    \]
   % \label{logarthimic gradient estimate}
    \item \textbf{Positivity for all times:} For a.e. $(t,x,m)\in [0,T_{\ast}]\times \mathbb{T}^{d}\times \mathbb{R}_{+}$ it holds that $
    \psi(t,x,m) >0$.     
\end{enumerate} 
\end{thm}
We do not need the full strength of assumption (\textbf{H2}) in order to establish local existence and uniqueness of \eqref{mean field PDE compact} -- in particular, the exponential decay assumption is needed only for the bounds on the logarithmic gradient. 
Next, we focus on the Kolmogorov equation \eqref{Kolmogorov eq}.
\begin{defn}[Weak solution to \eqref{Kolmogorov eq}]
 A function $\psi_{N}\in L^{\infty}([0,T];L^{1}(\mathbb{T}^{dN}\times \mathbb{R}_{+}^{N})\cap \mathscr{P}(\mathbb{T}^{dN}\times \mathbb{R}_{+}^{N}))$ is called a weak solution to  \eqref{Kolmogorov eq} with initial data $\psi_{N,0}\in C^{\infty}(\mathbb{T}^{dN}\times \mathbb{R}_{+}^{N})\cap W^{1,\infty}(\mathbb{T}^{dN}\times \mathbb{R}_{+}^{N})\cap W^{1,1}(\mathbb{T}^{dN}\times \mathbb{R}_{+}^{N})\cap \mathscr{P}(\mathbb{T}^{dN}\times \mathbb{R}_{+}^{N})$ if it fulfills the following conditions:
    \begin{enumerate}
     \item[(\textbf{E0})] \textbf{Weak formulation:} $\psi_{N}(0,\cdot)=\psi_{N,0}$ and for each $\varphi \in C^{\infty}_{0}(\mathbb{T}^{dN}\times \mathbb{R}_{+}^{N})$ it holds that 
    \begin{align*}
\frac{\dd}{\dd t}&\int_{\mathbb{T}^{dN}\times \mathbb{R}_{+}^{N}}\varphi(\mathbf{x}_{N},\mathbf{m}_{N})\psi_{N}(t,\mathbf{x}_{N},\mathbf{m}_{N})\ \dd \mathbf{x}_{N}\dd \mathbf{m}_{N}\\
&=\frac{1}{N}\sum_{i,j}\int_{\mathbb{T}^{dN}\times \mathbb{R}_{+}^{N}}m_{j}\mathbf{a}(x_{i}-x_{j})\nabla_{x_{i}}\varphi(\mathbf{x}_{N},\mathbf{m}_{N})\psi_{N}(t,\mathbf{x}_{N},\mathbf{m}_{N})\ \dd \mathbf{x}_{N}\dd \mathbf{m}_{N}\\
&+\frac{1}{N}\int_{\mathbb{T}^{dN}\times \mathbb{R}_{+}^{N}}\sum_{i,j}S(x_{i},m_{i},x_{j},m_{j})\psi_{N}(t,\mathbf{x}_{N},\mathbf{m}_{N})\partial_{m_{i}}\varphi(\mathbf{x}_{N},\mathbf{m}_{N})\ \dd \mathbf{x}_{N}\dd \mathbf{m}_{N}. 
    \end{align*} \item[\textbf{(E1)}] \textbf{Exponential decay:} There is some  $\mathbf{E}_{N}>0$ such that  
    \begin{align}
\left\Vert e^{\frac{1}{2}\sum_{k=1}^{N}m_{k}}\psi_{N}(t,\mathbf{x}_{N},\mathbf{m}_{N})\right\Vert_{\infty} \leq \mathbf{E}_{N}, \ \forall t\in[0,T].    \label{exp moments est}
\end{align}
    \end{enumerate} \label{def of entropy sol}
 \end{defn}

Only existence, and not uniqueness of weak solutions to the Kolmogorov equation can be proved. 
\begin{prop}
Let assumptions \textbf{H1} and \textbf{H3} hold and assume in addition that $\partial_{m}S\in W^{1,1}(\mathbb{T}^{d}\times \mathbb{R}_{+})$. Then, there exist a weak solution $\psi_{N}$ to \eqref{Kolmogorov eq} in the sense of Definition \ref{def of entropy sol} fulfilling the following entropy inequality  
\begin{align}
\int_{\mathbb{T}^{dN}\times\mathbb{R}_{+}^{N}}& \psi_{N}(t,\mathbf{x}_{N},\mathbf{m}_{N}) \log(\psi_{N}(t,\mathbf{x}_{N},\mathbf{m}_{N}))\  \dd \mathbf{x}_{N}\dd \mathbf{m}_{N} \notag\\
\leq&\, \int_{\mathbb{T}^{dN}\times\mathbb{R}_{+}^{N}} \psi_{N,0}(\mathbf{x}_{N},\mathbf{m}_{N})\log(\psi_{N,0}(\mathbf{x}_{N},\mathbf{m}_{N}))\ \dd \mathbf{x}_{N}\dd \mathbf{m}_{N} \notag\\
&-\frac{1}{N}\sum_{i,j }\int_{0}^{t} \int_{\mathbb{T}^{dN}\times \mathbb{R}^{N}_{+}} \partial_{m_{i}}S(x_{i},m_{i},x_{j},m_{j})\psi_{N}(\tau,\mathbf{x}_{N},\mathbf{m}_{N})\ \dd \mathbf{x}_{N}\dd \mathbf{m}_{N}\dd \tau,  \label{entropy}
\end{align}
for a.e. $t\in[0,T]$.
\label{existence for Kolmogorov}
\end{prop} 

Our main theorem is the following propagation of chaos result in the $L^1$ norm which is obtained by bounds on the relative entropy. Recall that we invoke the notation $\overline{\psi_{N}}=\psi^{\otimes N}$. 
\begin{thm}[Propagation of chaos]
\label{main thm}Let assumptions \textbf{H1}-\textbf{H3} hold and assume in addition that $\partial_{m}S\in W^{1,1}((\mathbb{T}^{d}\times \mathbb{R}_{+})^{2})$. Choose $\overline{S}_{0},\overline{S}_{1},\left\Vert \mathbf{a}\right\Vert_{\infty}$ and $T_{\ast\ast}>0$ small enough and let $\psi$ be the unique solution to \eqref{limit PDE Intro} on $[0,T_{\ast\ast}]$ with initial data $\psi_{0}$, as guaranteed via Theorem \ref{well posedness intro}. Let $\psi_{N}$ be a weak solution to \eqref{Kolmogorov eq} fulfilling the entropy inequality \eqref{entropy}, as guaranteed via Proposition \ref{existence for Kolmogorov}. Then there exist some $\delta>0$ such that 
\begin{align}
\mathcal{H}_{N}(t)\leq\left(\mathcal{H}_{N}(0)+\frac{T_{\ast\ast}}{N}\log\delta\right)e^{t}\  \ \mbox{for 
all}\ t\in [0,T_{\ast\ast}]. 
\label{gronwall for ent}
\end{align}
Consequently, for each $k\in \mathbb{N}$ we have convergence of the $k$-th marginal distributions in the following sense 
\begin{align}\label{final_L1 result}
\underset{t\in[0,T_{*\ast}]}{\sup}\left\Vert \psi_{N:k}(t,\cdot)-\psi^{\otimes k}(t,\cdot)\right\Vert _{L^{1}(\mathbb{T}^{dk}\times\mathbb{R}_{+}^{k})}\underset{N\rightarrow\infty}{\longrightarrow}0.    
\end{align}
\end{thm}
The explicit smallness assumptions on $\overline{S}_{0},\overline{S}_{1},\left\Vert \mathbf{a}\right\Vert_{\infty},T_{\ast \ast}$ needed for Theorem \ref{main thm} can be seen later in the proof, c.f \eqref{1st require} and \eqref{2nd require}. Note that in Theorem \eqref{main thm} the time interval $[0,T_{\ast\ast}]$ on which convergence holds may be shorter than the maximal time of existence of the underlying equation \eqref{limit PDE Intro}. 

We remark that \eqref{final_L1 result} is a consequence of \eqref{gronwall for ent} via the following \textit{CKP-inequality}, see \cite{villani2009optimal} for instance: 
Let $1\leq k\leq N$.
Let $\psi_{N}\in\mathscr{P}(\mathbb{T}^{dN}\times\mathbb{R}_{+}^{N})$
and let $\psi_{N:k}$ be its $k-$th marginal as defined in formula \eqref{k marginal}. Let $\psi\in\mathscr{P}(\mathbb{T}^{d})$. Then, it holds that 
\begin{equation}
\label{CPK inequality sec pre}    \left\Vert \psi_{N:k}-\psi^{\otimes k}\right\Vert _{L^1(\mathbb{T}^{dk}\times \mathbb{R}_{+}^{k})}\leq\sqrt{2k\mathcal{H}(\psi_{N}\left|\overline{\psi_{N}})\right.}.
\end{equation}
Our results are related to a question raised in \cite{pouradier2021mean}, regarding the possibility of deriving  the mean-field limit for weighted dynamics in the case where $S$ has jump discontinuities in $(x,y)$.  Observe that the method of \cite{pouradier2021mean} rests upon
a stability estimate for measure valued solutions \'a-la Dobrushin, i.e an estimate of the form $W_{1}(\mu(t,\cdot),\nu(t,\cdot))\leq CW_{1}(\mu^{0},\nu^{0})$
where the constant $C$ depends on the Lipschitz constant of $S$ with respect to $x,y$.
Therefore this method is inadequate for the case where $S$ has lower than Lipschitz regularity. Another recent approach dealing with time varying weights and singular $S$, introduced in \cite{ben2024graph}, proves the graph limit for \eqref{particles} with a continuous combination of Dirac deltas on position multiplied by functions of the weights as initial data. In particular, this implies the mean field limit for this specific type of initial data. It is also worth mentioning that the class of singular weight kernels $S$ in the present work is more general in comparison to \cite{ben2024graph} in terms of the regularity assumptions in the variables $(x,y)$ and that we impose fairly weak regularity assumptions on $\mathbf{a}$ as well.

Summarizing, our novel approach extends the relative entropy techniques of \cite{jabin2016mean} to time-varying weights with low regularity assumptions on the interaction and the influence kernels. We emphasize that
an important ingredient of the proof presented in \cite{jabin2016mean} is the control on the growth of the logarithmic gradient of the limiting transport equation \eqref{non local transport}. In this context, some of the technical highlights of the present work are the justification of the local well-posedness of the mean field limit PDE \eqref{limit PDE Intro} together with the proof of existence of a weak solution for the Kolmogorov equation \eqref{Kolmogorov eq} satisfying an entropy inequality. Furthermore, we derive new estimates of the form
\begin{align*}
\left\vert \partial_{m}\log(\psi)\right\vert\leq C 
\qquad
\mbox{and} 
\qquad
\left\vert \nabla_{x}\log(\psi)\right\vert\leq C,    
\end{align*}
which will be part of Section \ref{uniqueness exitsence etc}. These steps pave the way to a key calculation of the evolution of the relative entropy between the limit and the particle PDEs. To close our mean-field limit, it is apparent from \eqref{CPK inequality sec pre} that a Gr{\"o}nwall estimate on $\mathcal{H}_{N}(t)$ would be enough in order to prove $\psi_{N:k}\underset{N\rightarrow \infty}{\rightarrow} \psi^{\otimes k}$ in the $L^{1}$ norm. This is the aim of Section \ref{Sec mean field} where an adapted cancellation lemma is needed proved in subsection \ref{section_cancellation}.   

We emphasize that our current method cannot treat the original case $S(x,m,y,n) = m n \,s(x,y)$ mentioned in \cite{pouradier2021mean} as we require boundedness of $S$ in $(m,n)$ in the cancellation Lemma (Lemma \ref{cancellation rule}). This limitation appears to be less significant in other parts of the proof.

\section{Well-posedness of limiting PDE and existence of solutions to Kolmogorov equation}\label{uniqueness exitsence etc}
This section is devoted to the proofs of Theorem \ref{well posedness intro} and Proposition \ref{existence for Kolmogorov}.
\subsection{Notation and preliminaries}
For brevity, we set $\Omega\coloneqq \mathbb{T}^{d}\times \mathbb{R}_{+}$ and $\Omega_{N}\coloneqq\mathbb{T}^{dN}\times\mathbb{R}_{+}^{N}$. For each $\psi_{N}\in\mathscr{P}(\Omega_{N})$
we introduce the following notations 
\begin{align}\label{definition_a_N}
\mathbf{a}_{N}\psi_{N}(\mathbf{x}_{N},\mathbf{m}_{N})\coloneqq\frac{1}{N}\sum_{i=1}^{N}\textnormal{div}_{x_{i}}\bigg(\sum_{j=1}^{N}m_{j}\mathbf{a}(x_{i}-x_{j})\psi_{N}(\mathbf{x}_{N},\mathbf{m}_{N})\bigg),    
\end{align}
\begin{align}\label{definition_S_N}
S_{N}\psi_{N}(\mathbf{x}_{N},\mathbf{m}_{N})\coloneqq\frac{1}{N}\sum_{i=1}^{N}\partial_{m_{i}}\bigg(\sum_{j=1}^N S(x_{i},m_{i},x_{j},m_{j})\psi_{N}(\mathbf{x}_{N},\mathbf{m}_{N})\bigg)    
\end{align}
and
$L_{N}\psi_{N}\coloneqq(\mathbf{a}_{N}+S_{N})\psi_{N}$.
With this notation  \eqref{Kolmogorov eq} takes the concise form 
\begin{align}
\partial_{t}\psi_{N}+L_{N}\psi_{N}=0.\label{Kolmogorov eq sec 2}    
\end{align}
Similarly, for the limiting system, we define for each $\psi\in\mathscr{P}(\Omega)$
\begin{align}
&\mathbf{a}_{\infty}\psi(x,m)\coloneqq\mathrm{\mathrm{div}}_{x}\left(\psi(x,m)\mathbf{a}\star \mu[\psi](x,m)\right), \label{def of ainfty}\\
&S_{\infty}\psi(x,m)\coloneqq\partial_{m}\left(\psi(x,m)\int_{\Omega}S(x,m,y,n)\psi(y,n)\ \dd y \dd n\right) = \partial_{m}\left(\psi(x,m) \mathbf{S}[\psi](x,m)\right) \label{def of Sinfty} 
\end{align}
and 
$L_{\infty}\coloneqq\mathbf{a}_{\infty}+S_{\infty}.$
With this notation \eqref{limit PDE Intro} writes 
\begin{equation}
\partial_{t}\psi+L_{\infty}\psi=0. \label{mean field PDE compact}
\end{equation}
The well-posedness theory of \eqref{Kolmogorov eq sec 2} and \eqref{mean field PDE compact}
is studied in detail in the next section. We observe that if $\psi$
is a solution of  \eqref{mean field PDE compact} then the tensor product $\overline{\psi_{N}}=\psi^{\otimes N}\in \mathscr{P}(\Omega^{N})$
is governed by an equation comparable to \eqref{Kolmogorov eq sec 2}, namely we
have the following Proposition. 
\begin{prop}
Let $\psi$ be a solution in the sense of Theorem \ref{well posedness intro} to \eqref{mean field PDE compact} with $\psi >0$. Then 
\begin{align} \label{Kolmogorov product}
     \partial_{t}\overline{\psi_{N}}+L_{N}\overline{\psi_{N}}=(\mathcal{R}_{N}+\mathcal{S}_{N})\overline{\psi_{N}} 
\end{align}
where 
\begin{align}\label{definition_R_N}
\mathcal{R}_{N}\coloneqq\frac{1}{N}\sum_{i=1}^{N}\stackrel[j=1]{N}{\sum}\left(m_{j}\mathbf{a}(x_{i}-x_{j})-\mathbf{a}\star \mu[\psi](t,x_{i})\right)\nabla_{x_{i}}\log(\psi(t,x_{i},m_{i}))   
\end{align}
and 
\begin{align}
\mathcal{S}_{N}\coloneqq&\frac{1}{N}\sum_{i=1}^{N}\sum_{j=1}^N(\partial_{m_{i}}S(x_{i},m_{i},x_{j},m_{j})+S(x_{i},m_{i},x_{j},m_{j})\partial_{m_{i}}\log(\psi(x_{i},m_{i}))\notag \\
&-\partial_{m}\mathbf{S}[\psi](x_{i},m_{i}) -\partial_{m_{i}}\log(\psi(x_{i},m_{i})\mathbf{S}[\psi](x_i,m_i)). \label{S reminder}
\end{align}
\label{eq for overlinepsiN}
\end{prop}
\begin{proof}
Using the  notation \eqref{def of ainfty}, \eqref{def of Sinfty} we compute that 
\[
\frac{\partial}{\partial t}\overline{\psi_{N}}=\stackrel[i=1]{N}{\sum}\partial_{t}\psi(x_{i},m_{i})\underset{j\neq i}{\prod}\psi(x_{j},m_{j})=-\stackrel[i=1]{N}{\sum}\left(\mathbf{a}_{\infty}\psi+S_{\infty}\psi\right)(x_{i},m_{i})\frac{\overline{\psi_{N}}}{\psi(x_{i},m_{i})}.
\]
Comparing $\mathbf{a}_N$ with $a_{\infty}$, we get
\begin{align*}
&\mathbf{a}_{N}\overline{\psi_{N}}-\stackrel[i=1]{N}{\sum}\left(\mathbf{a}_{\infty}\psi\right)(x_{i},m_{i})\frac{\overline{\psi_{N}}}{\psi(x_{i},m_{i})}\\
&=\frac{1}{N}\sum_{i=1}^{N}\stackrel[j=1]{N}{\sum}m_{j}\mathbf{a}(x_{i}-x_{j})\nabla_{x_{i}}\psi(x_{i},m_{i})\frac{\overline{\psi_{N}}}{\psi(x_{i},m_{i})}-\stackrel[i=1]{N}{\sum}\mathbf{a}\star\mu[\psi](x_{i})\nabla_{x_{i}}\psi(x_{i},m_{i})\frac{\overline{\psi_{N}}}{\psi(x_{i},m_{i})}\\
&=\frac{1}{N}\sum_{i=1}^{N}\stackrel[j=1]{N}{\sum}m_{j}\mathbf{a}(x_{i}-x_{j})\nabla_{x_{i}}\log(\psi(x_{i}, m_i))\overline{\psi_{N}}-\stackrel[i=1]{N}{\sum}\mathbf{a}\star\mu[\psi](x_{i})\nabla_{x_{i}}\log(\psi(x_{i}, m_i))\overline{\psi_{N}}\\
&=\frac{1}{N}\sum_{i=1}^{N}\stackrel[j=1]{N}{\sum}\left(m_{j}\mathbf{a}(x_{i}-x_{j})-\mathbf{a}\star\mu[\psi](x_{i})\right)\nabla_{x_{i}}\log(\psi(x_{i},m_i))\overline{\psi_{N}}\\
&=\mathcal{R}_{N}\overline{\psi_{N}}.
\end{align*}
Secondly, we compare $S_{N}$ and $S_{\infty}$  
\begin{align*}
S_{N}&\overline{\psi_{N}}-\stackrel[i=1]{N}{\sum}(S_{\infty}\psi)(x_{i},m_{i})\frac{\overline{\psi_{N}}}{\psi(x_{i},m_{i})}
=\frac{1}{N}\sum_{i=1}^{N}\partial_{m_{i}}\bigg(\sum_{j=1}^NS(x_{i},m_{i},x_{j},m_{j})\overline{\psi_{N}}\bigg)\\&-\stackrel[i=1]{N}{\sum}\partial_{m_{i}}\left(\psi(x_{i},m_{i})\mathbf{S}[\psi](x_i,m_i)\right)\frac{\overline{\psi_{N}}}{\psi(x_{i},m_{i})}\\
=&\frac{1}{N}\sum_{i=1}^{N}\sum_{j=1}^{N}\left(\partial_{m_{i}}S(x_{i},m_{i},x_{j},m_{j})-\partial_{m} \mathbf{S}[\psi](x_i,m_i)\right)\overline{\psi_{N}}\\
&+\frac{1}{N}\sum_{i=1}^{N}\sum_{j=1}^N\bigg[S(x_{i},m_{i},x_{j},m_{j})-\mathbf{S}[\psi](x_i,m_i)\bigg]\partial_{m_{i}}\log(\psi(x_{i},m_{i}))\overline{\psi_{N}}=\mathcal{S}_{N}\overline{\psi_{N}}
\end{align*}
where we used that $\frac{\partial_{m_{i}}\psi}{\psi}=\partial_{m_{i}}\log(\psi)$.
    \end{proof}
\subsection{Well-posedness of limiting PDE} 
We start by establishing preliminary estimates  showing the propagation of moments property and estimates on the growth of the logarithmic gradient of a solution to \eqref{mean field PDE compact} (Theorem \ref{well posedness intro} 3.). These estimates will play an important role for the mean-field limit. 
We start with a general Lemma  providing upper and lower bounds on the supremum and infimum of solutions (or subsolutions) to a linear transport equation. 
\begin{lem}\label{linear transport lemma}
Let $\mathbf{A}\in C^{\infty}([0,T]\times \mathbb{T}^{d})$ and let $\mathbf{S},\mathbf{S}',\mathbf{S}''\in C^{\infty}_{0}([0,T]\times \Omega )$. Let $\Psi$ be a smooth subsolution to the following linear transport equation
\begin{align}
\partial_{t}\Psi+\nabla_{x}\Psi\cdot \mathbf{A}+\mathbf{S}\partial_{m}\Psi+\mathbf{S}'\Psi+\mathbf{S}''\leq 0, \ \Psi(0,x,m)=\Psi_{0}. \label{linear transport eq}    
\end{align}
Then, 
\begin{itemize}
    \item [i.]  For all $t\in [0,T]$ we have 
\begin{align*}
\underset{(x,m)\in \Omega}{\sup}\Psi(t,x,m)\leq e^{\left\Vert \mathbf{S}'\right\Vert_{\infty}T}\left(\underset{(x,m)\in \Omega}{\sup}\Psi_{0}(x,m)+T\left\Vert \mathbf{S}''\right\Vert_{\infty}\right).    
\end{align*}
\item[ii.] If in addition $\mathbf{S''}\equiv 0$ and $\Psi$ is a smooth solution, then for all $t\in [0,T]$ we have 
\begin{align*}
\inf_{(x,m)\in \Omega}\Psi(t,x,m)\geq e^{-\left\Vert \mathbf{S}'\right\Vert_{\infty}T}\underset{(x,m)\in \Omega}{\inf}\Psi_{0}(x,m). 
\end{align*}
\end{itemize}

\end{lem}
\begin{proof}
\textbf{Step 1}. For any $n\in \mathbb{N}$ let $\chi_{n}:\mathbb{R}_{+}\rightarrow [0,\infty)$ be a smooth cut-off function such that: 
\begin{itemize}
\item $\chi_{n}\in C^{\infty}(\mathbb{R}_{+})$.
\item $\mathrm{supp}(\chi_{n})\subset [\frac{1}{2n},2n] $.
\item $0\leq \chi_{n}\leq 1$ and $\chi_{n}\equiv 1$ on $[\frac{1}{n},n]$. 
\end{itemize}
Multiplying \eqref{linear transport eq} by $\chi_{n}$ we arrive at: 
\begin{align*}
\partial_{t}(\Psi\chi_{n})+\nabla_{x}(\chi_{n}\Psi)\cdot \mathbf{A}+\mathbf{S}\partial_{m}(\chi_{n}\Psi)-\partial_{m}\chi_{n}\Psi\mathbf{S}+\mathbf{S}'(\chi_{n}\Psi)+\mathbf{S''}\chi_{n}\leq 0.     
\end{align*}
Setting $\Psi_{n}\coloneqq \chi_{n}\Psi$ the above inequality writes 
\begin{align}
\partial_{t}\Psi_{n}+\nabla_{x}\Psi_{n}\cdot\mathbf{A}+\mathbf{S}\partial_{m}\Psi_{n}-\partial_{m}\chi_{n}\Psi \mathbf{S}+\mathbf{S}'\Psi_{n}+\mathbf{S}''\chi_{n}\leq 0. \label{equation for Psin}   
\end{align}
Since $\Psi_{
n}(t,\cdot)$ is a smooth compactly supported function, for any $t\in [0,T]$ it attains a global maximum $(x_{\max}^{t},m_{\max}^{t})\in \mathbb{T}^{d}\times (0,\infty)$. Evaluating \eqref{equation for Psin} at $(x_{\max}^{t},m_{\max}^{t})$ and using that $\nabla_{x}\Psi_{n}(
x_{\max}^{t},m_{\max}^{t})=0$ and $\partial_{m}\Psi_{n}(x^{t}_{\max},m_{\max}^{t})=0$ we get 
\begin{align*}
&\partial_{t}\Psi_{n}(t,x_{\max}^{t},m_{\max}^{t})-\partial_{m}\chi_{n}\Psi\mathbf{S}(t,x_{\max}^{t},m_{
\max
}^{t})\\
&+(\mathbf{S}'\Psi_{n})(t,x^{t}_{\max},m_{\max}^{t})+(\mathbf{S}''\chi_{n})(t,x_{\max}^{t},m_{\max}^{t})=0.     
\end{align*}
Since $\mathrm{supp}(\partial_{m}\chi_{n})\subset [\frac{1}{2n},\frac{1}{n}]\bigcup[n,2n]$ and $\mathbf{S}$ is compactly supported in $m$, it follows that for $n$ sufficiently large  we have $\partial_{m}\chi_{n}\mathbf{S}\equiv0$. Thus, we get 
\begin{align*}
\frac{\de}{\de t}\left\Vert \Psi_{n}(t,\cdot)\right\Vert_{\infty}+\mathbf{S}'(t,x_{\max}^{t},m_{\max}^{t})\left\Vert \Psi_{n}(t,\cdot)\right\Vert_{\infty}+(\mathbf{S}''\chi_{n})(t,x_{\max}^{t},m_{\max}^{t})=0.       
\end{align*}
Therefore we conclude the estimate  
\begin{align*}
\frac{\de}{\de t}\left\Vert \Psi_{n}(t,\cdot)\right\Vert_{\infty}\leq \left\Vert \mathbf{S}'\right\Vert_{\infty}\left\Vert \Psi_{n}(t,\cdot)\right\Vert_{\infty}+\left\Vert \mathbf{S}''\right\Vert_{\infty},       
\end{align*}
from which it follows that 
\begin{align*}
\left\Vert \Psi_{n}(t,\cdot)\right\Vert_{\infty}\leq e^{\left\Vert \mathbf{S}'\right\Vert_{\infty}T}\left(\left\Vert \Psi_{n}(0,\cdot)\right\Vert_{\infty}+T\left\Vert \mathbf{S}''\right\Vert_{\infty}\right).    \end{align*}
Therefore for all  $(t,x,m)\in [0,T]\times \Omega$ we have 
\begin{align*}
\chi_{n}\Psi(t,x,m)\leq e^{\left\Vert \mathbf{S}'\right\Vert_{\infty}T}\left(\chi_{n}\Psi_{0}(x,m)+T\left\Vert \mathbf{S}''\right\Vert_{\infty}\right).    
\end{align*}
Letting $n\rightarrow \infty$ and using that $\chi_{n}\underset{n\rightarrow \infty}{\rightarrow} 1$ pointwise we find 
\begin{align*}
\underset{(x,m)\in \Omega}{\sup}\Psi(t,x,m)\leq e^{\left\Vert \mathbf{S}'\right\Vert_{\infty}T}\left(\underset{(x,m)\in \Omega}{\sup}\Psi_{0}(x,m)+T\left\Vert \mathbf{S}''\right\Vert_{\infty}\right).    
\end{align*}
\textbf{Step 2}. Fix $\eta>0$ and set $\Phi_{\eta}=\frac{1}{\eta+\Psi}$. Dividing the equation 
\begin{align*}
\partial_{t}\Psi+\nabla_{x}\Psi\cdot \mathbf{A}+\mathbf{S}\partial_{m}\Psi+\mathbf{S}'\Psi=0    
\end{align*}
by $(\Psi+\eta)^{2}$ we get 
\begin{align*}
-\partial_{t}\Phi_{\eta}-\nabla_{x}\Phi_{\eta}\cdot \mathbf{A}-\mathbf{S}\partial_{m}\Phi_{\eta}+\frac{\mathbf{S}'\Psi}{\Psi+\eta}\Phi_{\eta}=0,  
\end{align*}
which is recast as 
\begin{align*}
\partial_{t}\Phi_{\eta}+\nabla_{x}\Phi_{\eta}\cdot \mathbf{A}+\mathbf{S}\partial_{m}\Phi_{\eta}-\frac{\mathbf{S}'\Psi}{\Psi+\eta}\Phi_{\eta}=0.    
\end{align*}
Note that $\left\Vert \frac{\mathbf{S}'\Psi}{\Psi+\eta}\right\Vert_{\infty}\leq \left\Vert \mathbf{S}'\right\Vert_{\infty}$ and therefore by step 1 we obtain the inequality 
\begin{align*}
\underset{(x,m)\in \Omega}{\sup}\Phi_{\eta}(t,x,m)\leq e^{\left\Vert \mathbf{S}'\right\Vert_{\infty}T}\underset{(x,m)\in \Omega}{\sup}\Phi_{\eta}(0,x,m).    
\end{align*}
Since 
\begin{align*}
\underset{(x,m)\in \Omega}{\sup}\Phi_{\eta}(t,x,m)=\frac{1}{\underset{(x,m)\in \Omega}{\inf}\Psi(t,x,m)+\eta} \end{align*}
it follows that 
\begin{align*}
\inf_{(x,m)\in \Omega}\Psi(t,x,m)+\eta\geq e^{-\left\Vert \mathbf{S}'\right\Vert_{\infty}T}\left(\underset{(x,m)\in \Omega}{\inf}\Psi_{0}(x,m) +\eta\right). 
\end{align*}
Letting $\eta\rightarrow 0$ yields the announced inequality. 
\end{proof}
The estimates to follow are of apriori type, i.e. we assume that $\mathbf{a},S$ are smooth but establish estimates independent of this smoothness assumption, which will be eventually be removed. The global well-posedness of \eqref{limit PDE Intro} when $\mathbf{a},S$ enjoy enough regularity follows by classical methods and is summarized in the following Proposition. 
\begin{prop}\label{regular system well posed}
Let \textbf{H1}-\textbf{H2} hold and assume in addition that $\mathbf{a}\in C^{\infty}(\mathbb{T}^{d})$ and $S\in C^{\infty}_{0}(\Omega)$. Then, there exist a unique smooth solution to the Cauchy problem 
\begin{align*}
\partial_{t}\psi+\mathrm{div}_{x}(\psi\mathbf{a}\star \mu[\psi])+\partial_{m}(\psi\mathbf{S}[\psi])=0, \ \psi(0,\cdot)=\psi_{0}.  \end{align*}
\end{prop}
The proof of Proposition \ref{regular system well posed} is based on a standard contraction argument. We start by demonstrating that $\psi$ remains a probability measure, provided it is  initially  a probability measure. This fact is stated together with  estimates on the $L^{\infty}$ and $L^{1}$ norms of $\psi$ and  $\mu[\psi]$.
\begin{lem} 
\label{Positivity}
Let \textbf{H1}-\textbf{H2} hold and assume in addition that $\mathbf{a}\in C^{\infty}(\mathbb{T}^{d})$ and $S\in C^
{\infty}_{0}(\Omega)$. Let $\psi(t,\cdot)$ be a smooth solution to \eqref{mean field PDE compact}. Then $\psi(t,\cdot)\in \mathscr{P}(\Omega)$ and $\mu[\psi](t,\cdot)\in L^{1}(\mathbb{T}^{d})\cap L^{\infty}(\mathbb{T}^{d})$. Moreover, for all $t\in [0,T]$ we have the bounds 
\begin{itemize}
    \item[i.] \begin{align*}
\left\Vert \mu[\psi](t,\cdot)\right\Vert_{1}\leq \overline{\mu}+\overline{S}_{0}t. 
\end{align*}
\item[ii.] \begin{align*}
 \left\Vert \psi(t,\cdot)\right\Vert_{\infty}\leq e^{\overline{S}_{1}T}\left\Vert \psi_{0}\right\Vert_{\infty}  \mbox{and}\  \left\Vert \mu[\psi](t,\cdot)\right\Vert_{\infty}\leq \left(T\overline{S}_{0}e^{\overline{S}_{1}T}\left\Vert \psi_{0}\right\Vert_{\infty}+\overline{\mu}\right)e^{\overline{S}_{0}T}.    
\end{align*}
\end{itemize}
\end{lem}
\begin{proof}
i. Let $b(r)=\left\vert r \right\vert$. Then, by multiplying the equation by $\mathrm{sgn}(\psi)$ we obtain that $b(\psi)$ is a solution to \eqref{mean field PDE compact}, i.e. 
\begin{align*}
\partial_{t}b(\psi)+\mathrm{div}_{x}\left(b(\psi)\mathbf{A}[\psi]\right)+\partial_{m}\left(b(\psi)\mathbf{S}[\psi]\right)=0, \ b(\psi)=\psi_{0}.     
\end{align*}
In addition, since $\psi$ is a solution one has 
\begin{align*}
\partial_{t}\psi+\mathrm{div}_{x}\left(\psi\mathbf{A}[\psi]\right)+\partial_{m}\left(\psi\mathbf{S}[\psi]\right)=0, \ \psi=\psi_{0}.   
\end{align*}
Thus, letting $a(\psi)=b(\psi)-\psi \geq 0$ we conclude that 
\begin{align*}
\partial_{t}a(\psi)+\mathrm{div}_{x}(a(\psi)\mathbf{A}[\psi])+\partial_{m}(a(\psi)\mathbf{S}[\psi])=0, \ a(\psi)(0,\cdot)=0.     
\end{align*}
Integrating over $\Omega$ we find that 
\begin{align*}
\frac{\dd }{\dd t}\int_{\Omega}a(\psi)(t,x,m)\ \dd x \dd m=0,     
\end{align*}
so that 
\begin{align*}
 \int_{\Omega}a(\psi)(t,x,m)\ \de x\de m=0  
\end{align*}
and hence
$\psi \equiv \left\vert \psi \right\vert\geq 0$. Of course, this argument is slightly formal and can be made rigorous by a classical approximation argument of the sign function. 
Moreover, multiplying by $m$ equation \eqref{mean field PDE compact} and integrating in $m$ gives 
\begin{align}
\partial_{t}\mu[\psi]+\nabla_{x}\mu[\psi]\cdot\mathbf{A}[\psi]-\int_{\mathbb{R}_{+}}\psi\mathbf{S}[\psi](t,x,m)\ \dd m=0.\label{integration only in m}    
\end{align}
Integrating in $x$ we get 
\begin{align}
\frac{\de}{\de t}\int_{\mathbb{T}^{d}}\mu[\psi](t,x)\ \de x
%+\int_{\Omega}\mathrm{div}_{x}(\mu[\psi_{R}]\mathbf{A}_{R}[\psi_{R}])m\ \de x\de m
-\int_{\Omega}\psi\mathbf{S}[\psi](t,x,m)\ \de x\de m=0. \label{multiplication by m}    
\end{align}
Note that 
\begin{align*}
\int_{\Omega}\psi\mathbf{S}[\psi](t,x,m)\ \de x\de m\leq \overline{S}_{0}\int_{\Omega}\psi(t,x,m)\ \de x\de m=  \overline{S}_{0}.   
\end{align*}  
So in view of \eqref{multiplication by m} we find that 
\begin{align*}
\frac{\de}{\de t}\left\Vert \mu[\psi](t,\cdot)\right\Vert_{1}\leq \overline{S}_{0}   \end{align*}
and hence 
$$
\left\Vert \mu[\psi](t,\cdot)\right\Vert_{1}\leq \overline{\mu}+\overline{S}_{0}t\leq \overline{\mu}+\overline{S}_{0}T.$$\\ 
ii. We continue by estimating $\left\Vert \psi(t,\cdot)\right\Vert_{\infty}$. Writing  \eqref{mean field PDE compact} in an expanded form we get 
\begin{align*}
\partial_{t}\psi+\nabla_{x}\psi\cdot \mathbf{A}[\psi]+\mathbf{S}[\psi]\partial_{m}\psi+\partial_{m}\mathbf{S}[\psi]\psi=0.   \end{align*}
Note that $\left\Vert \partial_{m}\mathbf{S}[\psi](t,\cdot)\right\Vert_{\infty}\leq \overline{S}_{1}$
and therefore it follows that from Lemma \ref{linear transport lemma} that 
\begin{align}
\left\Vert \psi(t,\cdot)\right\Vert_{\infty}\leq e^{\overline{S}_{1}t}\left\Vert \psi_{0}\right\Vert_{\infty}. \label{Linfty bound on psi}       
\end{align}
To prove the $L^{\infty}$ bound for $\mu[\psi](t,\cdot)$ we evaluate \eqref{integration only in m} at a global maximum point $x=x
_{\max}^{t}$. Using that $\nabla_{x}\mu[\psi](t,x_{\max}^{t})=0$ we get 
\begin{align*}
\partial_{t}\mu[\psi](t,x_{\max}^{t})-\int_{\mathbb{R}_{+}}\psi(t,x_{\max}^{t},m)\mathbf{S}[\psi](t,x_{\max}^{t},m)\ \dd m=0.  \end{align*}
Therefore it follows that 
\begin{align}
&\frac{\de}{\de t}\left\Vert \mu[\psi](t,\cdot)\right\Vert_{\infty}\leq \overline{S}_{0}\int_{\mathbb{R}_{+}}\psi(t,x_{\max}^{t},m)\ \dd m \notag\\
&\leq \overline{S}_{0}\left(\int_{0}^{1}\psi(t,x_{\max}^{t},m)\ \dd m+\int_{1}^{\infty}m\psi(t,x_{\max}^{t},m)\ \dd m\right) \notag\\
&\leq \overline{S}_{0}\left(\left\Vert \psi(t,\cdot)\right\Vert_{\infty}+\left\Vert \mu[\psi](t,\cdot) \right\Vert_{\infty} \right).  \label{time derivative of Linfty norm of mu}     
\end{align}
Substituting \eqref{Linfty bound on psi} in \eqref{time derivative of Linfty norm of mu}  we obtain
\begin{align*}
\frac{\dd}{\dd t}\left\Vert \mu[\psi](t,\cdot)\right\Vert_{\infty}\leq \overline{S}_{0}\left(e^{\overline{S}_{1}T}\left\Vert \psi_{0}\right\Vert_{\infty}+\left\Vert \mu[\psi](t,\cdot)\right\Vert_{\infty} \right),     
\end{align*}
and hence we find that  
\begin{align*}
\left\Vert \mu[\psi](t,\cdot)\right\Vert_{\infty}\leq \left(T\overline{S}_{0}e^{\overline{S}_{1}T}\left\Vert \psi_{0}\right\Vert_{\infty}+\overline{\mu}\right)e^{\overline{S}_{0}T}.\end{align*}
\end{proof}
\begin{lem} \label{lemma truncated}
Let \textbf{H1}-\textbf{H2} hold and assume in addition that $\mathbf{a}\in C^{\infty}(\mathbb{T}^{d})$, $S\in C^
{\infty}_{0}(\Omega)$ and there is a constant $\mathfrak{M}_{\mathrm{in}}>0$ such that $\mathfrak{M}_{b,p}(\psi_{0})\leq \mathfrak{M}_{\mathrm{in}}$. Let $\psi(t,\cdot)$ be a smooth solution to \eqref{mean field PDE compact}.   
Then, it holds that 
\begin{align}\mathfrak{M}_{b,p}(t)\leq e^{\overline{\mathfrak{M}}t}\mathfrak{M}_{\mathrm{in}},\ \mbox{for all} \ t\in[0,T] 
\label{moment ine}
\end{align}
where $\overline{\mathfrak{M}}=(b p+1)\overline{S}_{0}+(p+1)\overline{S}_{1}$.
%\label{Moment estimate}
\end{lem}
 \begin{proof} 
We study the evolution of $\mathfrak{M}_{b,p}(t)$. 
\begin{align*}
    \frac{\de}{\de t}&\mathfrak{M}_{b,p}(t)\\
    =&-p\int_{\Omega} \left(1+\ 
  m ^{bp}\right)\psi^{p-1}(t,x,m)\mathrm{div}_{x}\left(\psi\mathbf{A}[\psi]\right)(t,x,m)\ \de x\de m\\
  &-p\int_{\Omega}  \left(1+\ 
  m ^{bp}\right)\psi^{p-1}(t,x,m)\partial_{m}\left(\psi\mathbf{S}[\psi]\right)(t,x,m) \ \de x\de m\\
  =&-\int_{\Omega}(1+m^{b p})\nabla_{x}\psi^{p}(t,x,m)\mathbf{A}[\psi](t,x,m)\ \de x\de m\\
  &-p\int_{\Omega}(1+m^{b p})\psi^{p}(t,x,m)\partial_{m}\mathbf{S}[\psi](t,x,m)\ \de x\de m\\&-\int_{\Omega}(1+m^{bp})\partial_{m}\psi^{p}(t,x,m)\mathbf{S}[\psi](t,x,m)\ \de x\de m\coloneqq    I+J+L.
\end{align*}
Integrating by parts $I$ and using that $\mathrm{div}_{x}(\mathbf{A}[\psi])=0$ we see that  
\begin{align*}
I=0. 
\end{align*}
In addition, note that we have 
\begin{align}
\left\vert J\right\vert\leq p\left\Vert \partial_{m}S \right\Vert_{\infty}\mathfrak{M}_{b,p}(t)\leq p\overline{S}_{1}\mathfrak{M}_{b,p}(t).  \label{est on Jr}\end{align}
To bound $L$, note that
\begin{align*}
\partial_{m}\left(1+m^{b p}\right)=bpm^{bp-1},     
\end{align*}
and therefore 
\begin{align*}
L=&-\int_{\Omega}b pm^{bp-1}\psi^{p}(t,x,m)\mathbf{S}[\psi](t,x,m)\ \de x\de m\\
&+\int_{\Omega}(1+m^{b p})\psi^{p}(t,x,m)\partial_{m}\mathbf{S}[\psi](t,x,m)\ \de x\de m.
\end{align*}
Thus, it follows from hypothesis \textbf{H1} that  
\begin{align}
\left\vert L\right\vert\leq \overline{S}_{0}\mathfrak{M}_{b,p}(t)+b p\overline{S}_{0}\mathfrak{M}_{b,p}(t)+\overline{S}_{1}\mathfrak{M}_{b,p}(t)\leq ((b p+1)\overline{S}_{0}+\overline{S}_{1})\mathfrak{M}_{b,p}(t). \label{est on L_{r}}\end{align}
Gathering \eqref{est on Jr} and \eqref{est on L_{r}} we find that 
\begin{align*}
 \frac{\de}{\de t}\mathfrak{M}_{b,p}(t)\leq  ((b p+1)\overline{S}_{0}+(p+1)\overline{S}_{1})\mathfrak{M}_{b,p}(t).   
\end{align*}
By Gr\"onwall's Lemma we obtain 
\begin{align*}
\mathfrak{M}_{b,p}(t)\leq e^{\overline{\mathfrak{M}}t}\mathfrak{M}_{b,p}(0)\leq e^{\overline{\mathfrak{M}}t}\mathfrak{M}_{\mathrm{in}}, 
 \end{align*}
 where 
\begin{align*}
\overline{\mathfrak{M}}= (bp+1)\overline{S}_{0}+(p+1)\overline{S}_{1}.  
\end{align*}
\\
\end{proof}
We proceed by establishing propagation of boundedness of the logarithmic
gradient in $m$. 
\begin{lem}
\label{log grad est weight}
Let \textbf{H1}-\textbf{H2} hold and assume in addition that $\mathbf{a}\in C^{\infty}(\mathbb{T}^{d})$ and $S\in C^
{\infty}_{0}(\Omega)$. Let $\psi(t,\cdot)$ be a smooth solution to \eqref{mean field PDE compact}. Then, for all $(t,x,m)\in[0,T]\times \Omega$ it holds that
\begin{align}
 &\frac{1}{\mathcal{C}}e^{-m}e^{-(\overline{S}_{0}+\overline{S}_{1})T}\leq \psi(t,x,m)\leq e^{-m}\mathcal{C}e^{(\overline{S}_{0}+\overline{S}_{1})T}\label{lower and upper bound for psiR}     
\end{align}
and 
\begin{align}
 \left\vert \partial_{m}\psi(t,x,m)\right\vert\leq e^{-m}K_{\ref{bound in Lemma on m der}} \ \mbox{with}\ K_{\ref{bound in Lemma on m der}}\coloneqq \mathcal{C}\left(1+T\overline{S}_{2}e^{(\overline{S}_{0}+\overline{S}_{1})T}\right)e^{\overline{S}_{1}T}.\label{bound in Lemma on m der}    
\end{align}
Consequently, for all $t\in [0,T]$ it holds that 
\begin{align}
\left\Vert \partial_{m}\log(\psi(t,\cdot))\right\Vert_{\infty}\leq K_{\ref{ine for partialmlogpsi}} \ \mbox{with} \ K_{\ref{ine for partialmlogpsi}}\coloneqq \mathcal{C}^{2}\left(1+\overline{S}_{2}Te^{(\overline{S}_{0}+\overline{S}_{1})T}\right)e^{(\overline{S}_{0}+2\overline{S}_{1})T}. \label{ine for partialmlogpsi}     
\end{align}
 
\label{logarthimic gradient estimate}
\end{lem}

\begin{proof}
\textbf{Step 1}. Multiplying \eqref{mean field PDE compact} by $e^{m}$ we get:
\begin{align*}
 \partial_{t}(e^{m}\psi)+\nabla_{x}(e^{m}\psi)\cdot \mathbf{A}[\psi]+\partial_{m}(\psi e^{m})\mathbf{S}[\psi]+e^{m}\psi\partial_{m}\mathbf{S}[\psi]-e^{m}\psi\mathbf{S}[\psi]=0.   
\end{align*}
Setting $\Psi=e^{m}\psi$ the above equation writes 
\begin{align}
\partial_{t}\Psi+\nabla_{x}\Psi\cdot \mathbf{A}[\psi]+\partial_{m}\Psi\mathbf{S}[\psi]+(\partial_{m}\mathbf{S}[\psi]-\mathbf{S}[\psi]
)\Psi=0. \label{equation for empsiR}     
\end{align}
Taking  $\mathbf{S}'=\partial_{m}\mathbf{S}[\psi]-\mathbf{S}[\psi]$, we note the estimate 
\begin{align*}
\left\Vert \mathbf{S}'(t,\cdot)\right\Vert_{\infty}\leq \left\Vert \partial_{m}\mathbf{S}[\psi](t,\cdot)\right\Vert
_{\infty}+\left\Vert \mathbf{S}[\psi](t,\cdot)\right\Vert_{\infty}\leq \overline{S}_{1}+\overline{S}_{0}.  
\end{align*}
Thanks to Lemma \ref{linear transport lemma}ii. we deduce that 
\begin{align}
\underset{(x,m)\in \Omega}{\inf}e^{m}\psi(t,x,m)\geq \frac{1}{\mathcal{C}}e^{-(\overline{S}_{0}+\overline{S}_{1})T}, \label{inf bound from below}  
\end{align}
which establishes the lower bound in \eqref{lower and upper bound for psiR}. Moreover, by means of Lemma \ref{linear transport lemma}i. we find that 
\begin{align}
\underset{(x,m)\in \Omega}{\sup}e^{m}\psi(t,x,m)\leq e^{(\overline{S}_{0}+\overline{S}_{1})T}\underset{(x,m)\in \Omega}{\sup}e^{m}\psi(0,x,m)\leq \mathcal{C}e^{(\overline{S}_{0}+\overline{S}_{1})T}. \label{est on empsi}      
\end{align}
This establishes the upper bound in \eqref{lower and upper bound for psiR}.\\ 
\textbf{Step 2}. Differentiating \eqref{mean field PDE compact} with respect to $m$ we obtain  
\begin{align*}
&\partial_{t}\partial_{m}\psi+\nabla_{x}\partial_{m}\psi\cdot \mathbf{A}[\psi]+\partial_{mm}\psi\mathbf{S}[\psi]+2\partial_{m}\psi\partial_{m}\mathbf{S}[\psi]+\psi\partial_{mm}\mathbf{S}[\psi]=0.   
\end{align*}
Multiplying the above equation by by $\mathbf{s}(t,x)\coloneqq\mathrm{sgn}(\partial_{m}\psi(t,x))$ we get 
\begin{align*}
&\partial_{t}\vert\partial_{m}\psi\vert+\nabla_{x}\vert\partial_{m}\psi\vert\cdot \mathbf{A}[\psi]\\
&+\partial_{m}\vert\partial_{m}\psi\vert\partial_{m}\mathbf{S}[\psi]+2\left\vert \partial_{m}\psi\right\vert \partial_{m}\mathbf{S}[\psi]+\mathbf{s}\psi\partial_{mm}\mathbf{S}[\psi]=0.       
\end{align*}
Setting $\widetilde{\Psi}\coloneqq e^{m}\vert\partial_{m}\psi\vert$ and multiplying the above equation by $e^{m}$
we obtain 
\begin{align}
&\partial_{t} \widetilde{\Psi}+\nabla_{x} \widetilde{\Psi}\cdot \mathbf{A}[\psi]+\partial_{m}\widetilde{\Psi}\partial_{m}\mathbf{S}[\psi] \notag\\
&+\widetilde{\Psi}\partial_{m}\mathbf{S}[\psi]+\mathbf{s}e^{m}\psi\partial_{mm}\mathbf{S}[\psi]=0. \label{evolution of PsiRepsilon}   \end{align}
Taking $\mathbf{S}'=\partial_{m}\mathbf{S}[\psi]$ and $\mathbf{S}''=\mathbf{s}e^{m}\psi\partial_{mm}\mathbf{S}[\psi]$, we have the estimates 
\begin{align*}
\left\Vert \mathbf{S}'(t,\cdot)\right\Vert_{\infty}\leq\overline{S}_{1},\ \left\Vert \mathbf{S}''(t,\cdot)\right\Vert_{\infty}\leq  \mathcal{C}\overline{S}_{2}e^{(\overline{S}_{0}+\overline{S}_{1})T}      
\end{align*}
where the second inequality is due to \eqref{est on empsi}. By means of Lemma \ref{linear transport lemma}i. we obtain 
\begin{align}
\sup_{(x,m)\in \Omega}e^{m}\left\vert \partial_{m}\psi\right\vert(t,x,m)&\leq e^{\overline{S}_{1}T}\left(\sup_{(x,m)\in \Omega}e^{m}\vert\partial_{m}\psi_{0}\vert(x,m)+T\mathcal{C}\overline{S}_{2}e^{(\overline{S}_{0}+\overline{S}_{1})T} \right) \notag\\
&\leq e^{\overline{S}_{1}T}\mathcal{C}\left(1+T\overline{S}_{2}e^{(\overline{S}_{0}+\overline{S}_{1})T}\right). \label{sup bound from above}  \end{align}
To conclude, combining  \eqref{inf bound from below} and \eqref{sup bound from above} we obtain 
\begin{align*}
\left\Vert \partial_{m}\log(\psi(t,\cdot)
)\right\Vert_{\infty}= \left\Vert \frac{\partial_{m}\psi(t,\cdot)}{\psi(t,\cdot)}\right\Vert_{\infty}\leq\mathcal{C}^{2}\left(1+\overline{S}_{2}Te^{(\overline{S}_{0}+\overline{S}_{1})T}\right)e^{(\overline{S}_{0}+2\overline{S}_{1})T}.         
\end{align*}
\end{proof}
\begin{rem}
In particular, Lemma \ref{log grad est weight} shows that $\psi$  remains positive for all times,  a fact that will be used frequently in the sequel.     
\end{rem}
In the following Lemma we obtain uniform in time bounds on $\left\Vert \nabla_{x} \psi(t,\cdot)\right\Vert_{1},\left\Vert \nabla_{x} \mu[\psi](t,\cdot)\right\Vert_{1}$. This is the first estimate which is valid only on a short time interval. 
\begin{lem}\label{est on L1 nabla mu and nabla psi}
Let \textbf{H1}-\textbf{H2} hold and assume in addition that $\mathbf{a}\in C^{\infty}(\mathbb{T}^{d})$ and $S\in C^
{\infty}_{0}(\Omega)$. Let $\psi(t,\cdot)$ be a smooth solution to \eqref{mean field PDE compact}. Then, there is some $T_{\ast}>0$ such that it holds that 
\begin{align}
\left\Vert \nabla_{x}\psi(t,\cdot)\right\Vert_{1}+\left\Vert \nabla_{x} \mu(t,\cdot)\right\Vert_{1}\leq 2(1+\left\Vert \nabla_{x} \psi_{0}\right\Vert_{1}+\left\Vert \nabla_{x} \mu[\psi_{0}]\right\Vert_{1})  \label{est on gradmu and grad psi} \ \mbox{for all}\ t\in [0,T_{\ast}].   
\end{align}
\end{lem}
\begin{proof}
Taking the gradient in $x$ in \eqref{integration only in m} we get  
\begin{align*}
&\partial_{t}\nabla_{x}\mu[\psi]+D_{x}\nabla_{x}\mu[\psi]\mathbf{A}[\psi]+D_{x}\mathbf{A}[\psi]\nabla_{x}\mu[\psi]\\
&-\int_{\mathbb{R}_{+}}\nabla_{x}\psi(t,x,m)\mathbf{S}[\psi](t,x,m)\ \dd m-\int_{\mathbb{R}_{+}}\psi(t,x,m)\nabla_{x}\mathbf{S}[\psi](t,x,m)\ \dd m=0. 
\end{align*}
Multiplying the above equation by $\mathbf{s}\coloneqq\frac{\nabla_{x}\mu[\psi]}{\left\vert \nabla_{x}\mu[\psi]\right\vert }$ we get 
\begin{align*}
&\partial_{t}\left\vert \nabla_{x}\mu[\psi]\right\vert+\nabla_{x}\left\vert \nabla_{x}\mu[\psi]\right\vert\mathbf{A}[\psi]+D_{x}\mathbf{A}[\psi]\nabla_{x}\mu[\psi]\cdot \mathbf{s}\\
&-\int_{\mathbb{R}_{+}}\mathbf{s}\cdot \nabla_{x}\psi \mathbf{S}[\psi](t,x,m)\ \dd m-\int_{\mathbb{R}_{+}}\psi(t,x,m)\mathbf{s}\cdot\nabla_{x}\mathbf{S}[\psi](t,x,m)\ \dd m=0.       
\end{align*}
Integrating in $x$ we get 
\begin{align*}
&\frac{\de}{\de t}\left\Vert \nabla_{x} \mu[\psi](t,\cdot)\right\Vert_{1}+\int_{\Omega}D_{x}\mathbf{A}[\psi](t,x)\nabla_{x}\mu[\psi](t,x,m)\cdot \mathbf{s} \ \dd x\dd m\\
&-\int_{\Omega} \mathbf{s}\cdot\nabla_{x}\psi(t,x,m) \mathbf{S}[\psi](t,x,m)\ \dd x\dd m-\int_{\Omega}\psi(t,x,m)\mathbf{s}\cdot \nabla_{x}\mathbf{S}[\psi](t,x,m)\ \dd x\dd m=0.       
\end{align*}
Observe the elementary inequalities 
\begin{align}
\left\Vert D_{x}\mathbf{A}[\psi](t,\cdot)\right\Vert_{\infty}\leq \left\Vert \mathbf{a}\right\Vert_{\infty}\left\Vert \nabla_{x} \mu[\psi](t,\cdot)\right\Vert_{1}, \ \left\Vert \mathbf{S}[\psi](t,\cdot)\right\Vert_{\infty}\leq \overline{S}_{0}.  \label{elementary ine}  \end{align}
In addition, by \eqref{definition_gradient_xy} we have 
\begin{align}
\left\Vert \nabla_{x}\mathbf{S}[\psi](t,\cdot)\right\Vert_{\infty}\leq\overline{S}_{0} \left\Vert \nabla_{x} \psi(t,\cdot)\right\Vert_{1}.  \label{elementary ine 2}   \end{align}
It follows that 
\begin{align}
\frac{\de}{\de t}\left\Vert \nabla_{x} \mu[\psi](t,\cdot)\right\Vert_{1}\leq \left\Vert \mathbf{a}\right\Vert_{\infty} \left\Vert \nabla_{x}\mu[\psi](t,\cdot)\right\Vert_{1}^{2}+2\overline{S}_{0}\left\Vert \nabla_{x} \psi(t,\cdot)\right\Vert_{1}. 
\label{L1 norm of mu}
\end{align}
We proceed by studying the evolution of $\left\Vert \nabla \psi(t,\cdot)\right\Vert_{1}$.  Taking the gradient in $x$ in \eqref{mean field PDE compact} yields
\begin{align*}
&\partial_{t}\nabla_{x}\psi+D_{x}\mathbf{A}[\psi]\nabla_{x}\psi+D_{x}\nabla_{x}\psi\mathbf{A}[\psi]\\
&+\nabla_{x}\psi\partial_{m}\mathbf{S}[\psi]+\psi\nabla_{x}\partial_{m}\mathbf{S}[\psi]+\partial_{m}\nabla_{x}\psi\mathbf{S}[\psi]+\partial_{m}\psi\nabla_{x}\mathbf{S}[\psi]=0. 
\end{align*}
Multiplying by $\mathbf{s}(t,x,m)\coloneqq\frac{\nabla_{x} \psi(t,x,m)}{\left\vert \nabla_{x} \psi(t,x,m)\right\vert}$ implies 
\begin{align}
&\partial_{t}\vert\nabla_{x}\psi\vert+D_{x}\mathbf{A}[\psi]\nabla_{x}\psi\cdot \mathbf{s}+\nabla_{x}\left\vert\nabla_{x}\psi\right\vert\cdot \mathbf{A}[\psi] \notag\\
&+\mathbf{s}\cdot\nabla_{x}\psi\partial_{m}\mathbf{S}[\psi]+\psi\mathbf{s}\cdot\nabla_{x}\partial_{m}\mathbf{S}[\psi]+\partial_{m}\left\vert \nabla_{x}\psi\right\vert\mathbf{S}[\psi]+\partial_{m}\psi\nabla_{x}\mathbf{S}[\psi]\cdot \mathbf{s}=0. \label{equation for gradxpsi} 
\end{align}
Observe the elementary inequalities 
\begin{align*}
\left\Vert \partial_{m}\mathbf{S}[\psi](t,\cdot)\right\Vert_{\infty}\leq \overline{S}_{1},\ \left\Vert \nabla_{x}\partial_{m}\mathbf{S}[\psi](t,\cdot)\right\Vert_{\infty}\leq \overline{S}_{1}\left\Vert \nabla_{x} \psi(t,\cdot)\right\Vert_{1}.    \end{align*}
Together with \eqref{elementary ine}-\eqref{elementary ine 2}  we obtain 
\begin{align}
 \frac{\de}{\de t} \left\Vert \nabla_{x} \psi(t,\cdot)\right\Vert_{1}&\leq \left\Vert \mathbf
 {a}\right\Vert_{\infty} \left\Vert \nabla_{x} \mu[\psi](t,\cdot)\right\Vert_{1}\left\Vert \nabla_{x}\psi(t,\cdot)\right\Vert_{1} \notag\\
&+3\overline{S}_{1}\left\Vert \nabla_{x}\psi(t,\cdot)\right\Vert_{1}+\overline{S}_{0}\left\Vert \partial_{m}\psi(t,\cdot)\right\Vert_{1} \left\Vert \nabla_{x}\psi(t,\cdot)\right\Vert_{1} . \label{time derivative of grad psi}      
\end{align}  
Combining \eqref{L1 norm of mu} and \eqref{time derivative of grad psi} we get 
\begin{align*}
&\frac{\de}{\de t}\left(\left\Vert \nabla_{x}\mu[\psi](t,\cdot)\right\Vert_{1}+\left\Vert \nabla_{x}\psi(t,\cdot)\right\Vert_{1} \right)\\
&\leq \left(3\overline{S}_{1}+2\overline{S}_{0} +\overline{S}_{0}\left\Vert \partial_{m}\psi(t,\cdot)\right\Vert_{1} \right)\left\Vert \nabla_{x} \psi(t,\cdot)\right\Vert_{1}+\left\Vert \mathbf{a}\right\Vert_{\infty} \left\Vert \nabla_{x}\mu[\psi](t,\cdot)\right\Vert_{1}^{2}\\
&+\left\Vert \mathbf
 {a}\right\Vert_{\infty} \left\Vert \nabla_{x} \mu[\psi](t,\cdot)\right\Vert_{1}\left\Vert \nabla_{x}\psi(t,\cdot)\right\Vert_{1}\\
&\leq \left(\left\Vert \mathbf{a}\right\Vert_{\infty}+3\overline{S}_{1}+(2+K_{\ref{bound in Lemma on m der}})\overline{S}_{0} \right)\left\Vert \nabla_{x} \psi(t,\cdot)\right\Vert_{1}+2\left\Vert \mathbf{a}\right\Vert_{\infty} (\left\Vert \nabla_{x}\mu[\psi](t,\cdot)\right\Vert_{1}^{2}+\left\Vert \nabla_{x}\psi(t,\cdot) \right\Vert_{1}^{2}),  \end{align*}
where the last inequality is due to Lemma \ref{log grad est weight}. Thus, denoting $Y(t)\coloneqq 1+\left\Vert \nabla_{x}\mu[\psi](t,\cdot)\right\Vert_{1}+\left\Vert \nabla_{x}\psi(t,\cdot)\right\Vert_{1}$ we get 
\begin{align*}
\frac{\de Y(t)}{\de t}\leq \left(3\left\Vert \mathbf{a}\right\Vert_{\infty}+3\overline{S}_{1}+(2+K_{\ref{bound in Lemma on m der}})\overline{S}_{0} \right)Y^{2}(t).    
\end{align*}
Solving the above inequality yields 
\begin{align*}
Y(t)\leq \frac{1}{\frac{1}{Y(0)}-TK(T)}  \end{align*}
with
\begin{align*}
K(T)=K(\left\Vert \mathbf{a}\right\Vert_{\infty},\overline{S}_{0},\overline{S}_{1},\overline{S}_{2},\mathcal{C},T)=3\left\Vert \mathbf{a}\right\Vert_{\infty}+3\overline{S}_{1}+(2+K_{\ref{bound in Lemma on m der}})\overline{S}_{0}.   
\end{align*}
Since $\underset{T\rightarrow 0}{\lim}TK(T)=0$, there is some $T_{\ast}=T_{\ast}(\left\Vert \mathbf{a}\right\Vert_{\infty},\overline{S}_{0},\overline{S}_{1},\mathcal{C},\left\Vert \nabla_{x}\psi_{0}\right\Vert_{1},\left\Vert \nabla_{x}\mu[\psi_{0}]\right\Vert_{1})>0$ such that $T_{\ast}K(T_{\ast})\leq \frac{1}{2Y(0)}$. 
\end{proof}
\begin{lem}\label{nablax psi bound}
Let \textbf{H1}-\textbf{H2} hold  and assume in addition that $\mathbf{a}\in C^{\infty}(\mathbb{T}^{d})$ and $S\in C^{\infty}_{0}(\Omega)$.  Let $\psi(t,\cdot)$ be a smooth solution to \eqref{mean field PDE compact} and let $T_{\ast}>0$ be as in Lemma \ref{est on L1 nabla mu and nabla psi}. Then, for all $0<T\leq T_{\ast}$ and all $(t,x,m)\in[0,T]\times \Omega$ it holds that 
\begin{align}
 \left\vert \nabla_{x}\psi(t,x,m)\right\vert \leq  K_{\ref{grad x linfty est}}e^{-m}, \label{grad x linfty est}
 \end{align}
where 
\begin{align*}
K_{\ref{grad x linfty est}}\coloneqq&\exp\left((\overline{S}_{0}+\overline{S}_{1}+2\left\Vert \mathbf{a}\right\Vert_{\infty}(1+\left\Vert \nabla_{x} \psi_{0}\right\Vert_{1}+\left\Vert \nabla_{x} \mu[\psi_{0}]\right\Vert_{1}))T\right)\\
&\times \left(\mathcal{C}+2T\left(\overline{S}_{1}\mathcal{C}e^{(\overline{S}_{0}+\overline{S}_{1})T}  +\overline{S}_{0}K_{\ref{bound in Lemma on m der}}(T)\right)(1+\left\Vert \nabla_{x} \psi_{0}\right\Vert_{1}+\left\Vert \nabla_{x} \mu[\psi_{0}]\right\Vert_{1}  )\right).     
\end{align*}
Consequently, for all $0<T\leq T_{\ast}$ and $t\in [0,T]$ it holds that 
\begin{align*}
\left\Vert\nabla_{x}\log\psi(t,\cdot) \right\Vert_{\infty}\leq \mathcal{C}K_{\ref{grad x linfty est}}e^{(\overline{S}_{0}+\overline{S}_{1})T}.      
\end{align*}
\end{lem}
\begin{proof}
Multiplying \eqref{equation for gradxpsi} by $e^{m}$ we obtain 
\begin{align*}
&\partial_{t}(e^{m}\left\vert \nabla_{x}\psi\right\vert)+D_{x}\mathbf{A}[\psi]e^{m}\nabla_{x}\psi\cdot \mathbf{s}+\nabla_{x}(e^{m}\vert \nabla_{x}\psi\vert)\cdot \mathbf{A}[\psi]\\&+e^
{m}\left\vert \nabla_{x}\psi\right\vert\partial_{m}\mathbf{S}[\psi]
+ e^{m}\psi\nabla_{x}\partial_{m}\mathbf{S}[\psi]\cdot \mathbf{s}+\partial_{m}(e^{m}\left\vert\nabla_{x}\psi\right\vert)\mathbf{S}[\psi]\\
&-e^{m}\left\vert \nabla_{x}\psi\right\vert\mathbf{S}[\psi]+e^{m}\partial_{m}\psi\nabla_{x}\mathbf{S}[\psi]\cdot \mathbf{s}=0.  
\end{align*}
Setting $\Psi(t,x,m)\coloneqq e^{m}\left\vert \nabla_{x}\psi\right\vert(t,x,m)$ the latter equation leads to the inequality  
\begin{align*}
&\partial_{t}\Psi+\nabla_{x}\Psi\cdot\mathbf{A}[\psi]+\partial_{m}\Psi\mathbf{S}[\psi]\\
&+(\partial_{m}\mathbf{S}[\psi]-\left\Vert  D_{x}\mathbf{A}[\psi]  (t,\cdot)\right\Vert_{\infty}-\mathbf{S}[\psi]) \Psi\\
&+e^{m}\psi\nabla_{x}\partial_{m}\mathbf{S}[\psi]\cdot\mathbf{s}+e^{m}\partial_{m}\psi\nabla_{x}\mathbf{S}[\psi]\cdot\mathbf{s}\leq 0. 
\end{align*}
Let 
\begin{align*}
\mathbf{S}'\coloneqq\partial_{m}\mathbf{S}[\psi]-\left\Vert D_{x}\mathbf{A}[\psi](t,\cdot)\right\Vert_{\infty}-\mathbf{S}[\psi]     
\end{align*}
and 
\begin{align*}
\mathbf{S}''\coloneqq  e^{m}\psi\nabla_{x}\partial_{m}\mathbf{S}[\psi]\cdot \mathbf{s}+e^{m}\partial_{m}\psi_{}\nabla_{x}\mathbf{S}[\psi]\cdot \mathbf{s}.     
\end{align*}
To estimate $\mathbf{S}'$, we observe that 
\begin{align}
&\left\Vert \mathbf{S}'(t,\cdot)\right\Vert_{\infty} \leq  \left\Vert \partial_{m}\mathbf{S}[\psi](t,\cdot)\right\Vert_{\infty}+\left\Vert D_{x}\mathbf{A}[\psi](t,\cdot)\right\Vert_{\infty}+\left\Vert \mathbf{S}[\psi](t,\cdot)\right\Vert_{\infty}
\notag\\
&\leq \overline{S}_{1}+\left\Vert \mathbf{a} \right\Vert_{\mathrm{\infty}}\left\Vert \nabla_{x}\mu[\psi](t,\cdot)\right\Vert_{1}+\overline{S}_{0}\leq \overline{S}_{1}+2\left\Vert \mathbf{a}\right\Vert_{\infty}(1+\left\Vert \nabla_{x} \psi_{0}\right\Vert_{1}+\left\Vert \nabla_{x} \mu[\psi_{0}]\right\Vert_{1})+\overline{S}_{0}, \label{S' est for grad log x}         
\end{align}
where in the last inequality we invoked Lemma \ref{est on L1 nabla mu and nabla psi}. 
In addition, we have 
\begin{align*}
\left\vert \nabla_{x}\mathbf{S}[\psi](t,x,m)\right\vert=\left\vert \int_{\Omega}S(x,m,y,n)\nabla_{y}\psi(t,y,n)\ \dd y \dd n \right\vert\leq \overline{S}_{0}\left\Vert \nabla_{x}\psi(t,\cdot)\right\Vert_{1}       
\end{align*}
and 
\begin{align*}
\left\vert \nabla_{x}\partial_{m}\mathbf{S}[\psi](t,x,m)\right\vert=\left\vert \int_{\Omega}\partial_{m}S(x,m,y,n)\nabla_{y}\psi(t,y,n)\ \dd y\dd n\right\vert\leq \overline{S}_{1}\left\Vert \nabla_{x} \psi(t,\cdot)\right\Vert_{1}.    \end{align*}
Therefore, in view of  Lemma \ref{log grad est weight} and Lemma \ref{est on L1 nabla mu and nabla psi} we infer 
\begin{align}
\left\Vert \mathbf{S}''(t,\cdot)\right\Vert_{\infty} &\leq \left\Vert e^{m}\psi(t,\cdot)\right\Vert_{\infty}\left\Vert \nabla_{x}\partial_{m}\mathbf{S}[\psi](t,\cdot)\right\Vert_{\infty}+\left\Vert e^{m}\partial_{m}\psi(t,\cdot)\right\Vert_{\infty}\left\Vert \nabla_{x}\mathbf{S}[\psi](t,\cdot)\right\Vert_{\infty}\notag\\
& \leq \overline{S}_{1}\left\Vert e^{m}\psi(t,\cdot)\right\Vert_{\infty}\left\Vert \nabla_{x} \psi(t,\cdot)\right\Vert_{1}+\overline{S}_{0}\left\Vert e^{m}\partial_{m}\psi(t,\cdot)\right\Vert_{\infty}\left\Vert \nabla_{x} \psi(t,\cdot)\right\Vert_{1}    
\notag\\&\leq 2\left(\overline{S}_{1}\mathcal{C}e^{(\overline{S}_{0}+\overline{S}_{1})T} +\overline{S}_{0}K_{\ref{bound in Lemma on m der}}(T)\right)(1+\left\Vert \nabla \psi_{0}\right\Vert_{1}+\left\Vert \nabla \mu[\psi_{0}]\right\Vert_{1}  ).  \label{S'' est for grad log}     \end{align}
Gathering \eqref{S' est for grad log x}-\eqref{S'' est for grad log} and applying Lemma \ref{linear transport lemma}i. we conclude that for all $t\in [0,T]$ it holds that 
\begin{align}
\left\Vert e^{m}\nabla_{x}\psi(t,\cdot)\right\Vert_{\infty}\leq K_{\ref{grad x linfty est}}. \label{emnabla}      
\end{align}
Finally, we deduce from  \eqref{lower and upper bound for psiR} and \eqref{emnabla}  that 
\begin{align*}
\left\Vert \nabla_{x}\log(\psi)(t,\cdot)\right\Vert_{\infty}\leq \left\Vert \frac{\nabla_{x}\psi(t,\cdot)}{\psi(t,\cdot)}\right\Vert_{\infty}\leq \mathcal{C}K_{\ref{grad x linfty est}}e^{(\overline{S}_{0}+\overline{S}_{1})T}.  \end{align*}
\end{proof}
In order to remove the smoothness assumptions we apply a standard compactness argument. Let $S_{R}$ be a mollification in $(x,m,y,n)$ of $S$ and $\mathbf{a}_{R}$ a mollification of $\mathbf{a}$ such that 
\begin{align}
\left\Vert \partial_{m}^{\alpha}S_{R}\right\Vert_{\infty}\leq \left\Vert \partial_{m}^{\alpha}S\right\Vert_{\infty}\  \mbox{for} \ \alpha=0,1,2  \ \mbox{and}\ \left\Vert \mathbf{a}_{R}\right\Vert_{\infty}\leq \left\Vert \mathbf{a}\right\Vert_{\infty}. \label{mollification properties} 
\end{align}
We define
\begin{align*}
 \ \mathbf{A}_{R}[\psi](x)\coloneqq \mathbf{a}_{R}\star \mu[\psi](x)\in C^{\infty}(\mathbb{T}^{d}), ~~
\mathbf{S}_{R}[\psi](x,m)\coloneqq \int_{\Omega}S_{R}(x,m,y,n)\psi(y,n)\ \dd y\dd n\in C^{\infty}_{0}(\Omega),    
\end{align*}
where we recall the definition of $\mu[\psi]$ in \eqref{mudef}.
With this notation we consider the regularized equation  
\begin{align}\partial_{t}\psi_{R}+\mathrm{div}_{x}(\psi_{R} \mathbf{A}_{R}[\psi_{R}])+\partial_{m}(\psi_{R} \mathbf{S}_{R}[\psi_{R}])=0,\  \psi(0,\cdot)= \psi_{0},     \label{truncated limit eq}   
\end{align}
which is globally well posed for any fixed $R>0$ by Proposition \ref{regular system well posed}.
A particular consequence of Lemma \ref{Positivity}, Lemma \ref{logarthimic gradient estimate} and Lemma \ref{nablax psi bound} is the following propagation of Sobolev norms.   
\begin{cor}
\label{prop of sobolev norms}
Let \textbf{H1}-\textbf{H2} hold and let $\psi_{R}$ be a smooth solution to \eqref{truncated limit eq} on $[0,T_{\ast}]$ with $T_{\ast}>0$ as in Lemma \ref{est on L1 nabla mu and nabla psi}. Then, there is a constant $C>0$, independent of $R$, such that it holds that 
\begin{align*}
\left\Vert  \psi_{R}(t,\cdot)\right\Vert_{W  ^{1,1}(\Omega)\cap W^{1,\infty}(\Omega)}\leq C \ \mbox{and} \ \left\Vert \mu[\psi_{R}](t,\cdot)\right\Vert_{W^
{1,1}(\Omega)}\leq C.     
\end{align*}
\end{cor}
We can now prove existence for the limiting PDE \eqref{limit PDE Intro}.
\begin{thm}
Let $\mathbf{a}$ and $S$ satisfy the assumptions of \textbf{H1}, let $\psi_{0}$ satisfy the assumptions of \textbf{H2} and let $T_{\ast}>0$ be as in Lemma \ref{est on L1 nabla mu and nabla psi}. Then, there exists a solution $\psi\in C([0,T_{\ast}];L^{\infty}\cap L^{1}(\Omega))\cap L^{\infty}([0,T_{\ast}];W^{1,\infty}(\Omega)\cap W^{1,1}(\Omega))$ for the Cauchy problem  
\begin{align*}
\partial_{t}\psi+\mathrm{div}_{x}\left(\psi\mathbf{A}[\psi]\right)+\partial_{m}\left(\psi \mathbf{S}[\psi]\right)=0, \ \psi(0,x)=\psi_{0}.    
\end{align*}
Moreover, $\psi$ satisfies $\mu[\psi]\in L^{\infty}([0,T];W^{1,1}(\Omega))$. 
\end{thm}
\begin{proof}
Consider the solution $\psi_{R}$ of the regularized equation \eqref{truncated limit eq}. By Corollary \ref{prop of sobolev norms} there is a constant $C>0$, independent of $R$, such that for all $1\leq p\leq \infty$ it holds that 
\begin{align}
\ \underset{t \in [0,T_{\ast}]}{\sup}\left\Vert  \psi_{R}(t,\cdot)\right\Vert_{W^{1,p}(\Omega)}\leq C \ \mbox{and}\ \left\Vert \mu[\psi_{R}](t,\cdot)\right\Vert_{W^{1,1}(\Omega)}\leq C.    
\label{reminder of est for nabla   and partial}
\end{align}
In addition, by Lemma \ref{Positivity} we have the bounds 
\begin{align}
\left\Vert \mathbf{A}_{R}[\psi_{R}] (t,\cdot)\right\Vert_{\infty}\leq  \left\Vert \mathbf{a}\right\Vert_{\infty}\left\Vert \mu[\psi_{R}](t,\cdot)\right\Vert_{\infty}\leq  \left\Vert \mathbf{a}\right\Vert_{\infty}(1
+\overline{S}_{0}t), \  \left\Vert \mathbf{S}_{R}[\psi_{R}] \right\Vert_{\infty}\leq \overline{S}_{0}.    \label{A S est}   
\end{align}
Therefore, from \eqref{reminder of est for nabla and partial},\eqref{A S est} we obtain  
\begin{align}
&\left\Vert  \psi_{R}(t,\cdot)-\psi_{R}(s,\cdot)\right\Vert_{p}\notag\\&\leq \sup_{t\in [0,T_{\ast}]}\Big(\left\Vert \nabla \psi_{R}(t,\cdot)\right\Vert_{p}\left\Vert \mathbf{A}_{R}[\psi_{R}](t,\cdot)\right\Vert_{\infty}+\left\Vert \partial_{m}\mathbf{S}_{R}[\psi_{R}](t,\cdot)\right\Vert_{\infty}\left\Vert \psi_{R}(t,\cdot)\right\Vert_{p} \notag\\
&+\left\Vert \partial_{m}\psi_{R}(t,\cdot)\right\Vert_{p}\left\Vert \mathbf{S}_{R}[\psi_{R}](t,\cdot) \right\Vert_{\infty}\Big)\left\vert t-s\right\vert \leq C\left\vert t-s\right\vert, 
\label{arzela asoli}   \end{align}
where $C>0$. By the Arzela-Ascoli Theorem, there exists a subsequence $R_{n}$ and some $\psi \in C([0,T];L^{p}(\mathbb{R}^{d}))$ such that $\left\Vert \psi_{R_{n}}-\psi \right\Vert_{C([0,T];L^p(\mathbb{R}^{d}))}\underset{n \rightarrow \infty}{\rightarrow} 0$. Furthermore by the Banach-Alaoglu theorem and \eqref{reminder of est for nabla   and partial} it follows that $\psi \in L^{\infty}([0,T];W^{1,p}(\mathbb{R}^{d}))$ 
and that $\mu[\psi]\in L^{\infty}([0,T];W^{1,1}(\Omega))$. It is readily checked that the limit point $\psi$ is the asserted solution.  
 \end{proof}
The last part of this section is devoted to show uniqueness of solutions to \eqref{limit PDE Intro}. Recall from the previous section that $\mu[\psi]=\int_{\mathbb{R}_{+}}n\psi(x,n)\ \de n$ is governed by the equation 
\begin{equation}
\partial_{t}\mu[\psi]+\mathrm{div}_{x}\left(\mu[\psi] \mathbf{a}\star \mu[\psi] \right)=h[\psi], \label{R^d eq}  
\end{equation}
where 
\begin{align*}
 h\left[\psi\right](x)=\int_{\mathbb{R}^{d}}S(x,m,y,n)\psi(y,n)\psi(x,m) \ \de y\de n\de m.
\end{align*}
\begin{thm}
Let $\mathbf{a}$ and $S$ satisfy assumption \textbf{H1}. Let $\psi_{1},\psi_{2}$ be strong solutions to \eqref{limit PDE Intro} on $[0,T]$
with initial data $\psi_{0}^{1},\psi_{0}^{2}$. 
Then, there is some constant $C>0$ such that 
\begin{align*}
\left\Vert (\psi_{1}-\psi
_{2})(t,\cdot)\right\Vert_{1}+\left\Vert (\mu[\psi_{1}]-\mu[\psi_{2}])(t,\cdot)\right\Vert_{1}\leq e^{Ct}\left(\left\Vert \psi_{0}^{1}-\psi_{0}^{2}\right\Vert_{1}+\left\Vert \mu[\psi_{0}^{1}]-\mu[\psi_{0}^{2}]\right\Vert_{1} \right)       
\end{align*}
for all $t\in [0,T]$. Consequently, uniqueness of strong solutions to \eqref{limit PDE Intro} follows. 
\label{uniqueness}
\end{thm}
\textit{Proof}. 
\textbf{Step 1}. \textit{Calculation of $\frac{\de}{\de t}\left\Vert (\mu[\psi_{1}]-\mu[\psi_{2}])(t,\cdot)\right\Vert_{1}$}. 
By equation \eqref{R^d eq} we have 
\begin{align*}
\frac{\de}{\de t}\left\Vert \mu_{1}(t,\cdot)-\mu_{2}(t,\cdot)\right\Vert _{1}
=&\, \int_{\mathbb{T}^{d}} \left(\mathrm{div}_{x}(\mu_{2}\mathbf{a}\star\mu_{2})-\mathrm{div}_{x}(\mu_{1}\mathbf{a}\star\mu_{1})\right)(t,x)\mathrm{sgn}(\mu_{1}(t,x)-\mu_{2}(t,x))\ \de x \\
&+\int_{\mathbb{T}^{d}}( h\left[\psi_{1}\right]-h\left[\psi_{2}\right])(t,x) \mathrm{sgn}(\mu_{1}(t,x)-\mu_{2}(t,x))\ \de x =I_{1}+I_{2}.
\end{align*}
Put $\mathrm{s}(t,x)\coloneqq\mathrm{sgn}(\mu_{1}(t,x)-\mu_{2}(t,x))$. We start by estimating $I_{1}$. Using that $\mathrm{div}_{x}(\mathbf{a})=0$ we have 
\begin{align}
\int_{\mathbb{T}^{d}}&\left(\mathrm{div}_{x}(\mu_{2}\mathbf{a}\star\mu_{2})-\mathrm{div}(\mu_{1}\mathbf{a}\star\mu_{1})\right)(t,x)\mathrm{s}(t,x)\ \de x \notag\\
=&\int_{\mathbb{T}^{d}}\mathrm{div}_{x}((\mu_{2}-\mu_{1})\mathbf{a}\star\mu_{2})(t,x)\mathrm{s}(t,x)\ \de x \notag
+\int_{\mathbb{T}^{d}}\mathrm{div}_{x}(\mu_{1}\mathbf{a}\star(\mu_{2}-\mu_{1}))(t,x)\mathrm{s}(t,x)\ \de x \notag\\
=&-\int_{\mathbb{T}^{d}}\nabla_{x}\left\vert\mu_{2}-\mu_{1}\right\vert(t,x)\mathbf{a}\star\mu_{2}(t,x)\ \de x
+\int_{\mathbb{T}^{d}}\nabla_{x}\mu_{1}(t,x)\mathbf{a}\star(\mu_{2}-\mu_{1})(t,x)\mathrm{s}(t,x)\ \de x. 
\label{I1}
\end{align}
Integrating by parts the first term and using that $\mathrm{div}_{x}(\mathbf{a})=0$, we find that the first term in the right-hand side of \eqref{I1} vanishes, and thus  
\begin{align*}
I_{1}=&\int_{\mathbb{T}^{d}}\nabla_{x}\mu_{1}(t,x)\mathbf{a}\star(\mu_{2}-\mu_{1})(t,x)\mathrm{s}(t,x)\ \de x\leq \left\Vert \nabla_{x} \mu_{1}\right\Vert_{L^{\infty}_{t}L^{1}_{x}}\left\Vert \mathbf{a}\right\Vert_{\infty}\left\Vert (\mu_{1}-\mu_{2})(t,\cdot)\right\Vert_{1}, 
\end{align*}
and so we get 
\begin{align}
I_{1}\lesssim \left\Vert (\mu_{1}-\mu_{2})(t,\cdot)\right\Vert_{1}. \label{est for I1}
\end{align}
Next we estimate $I_{2}$. First, note that we have 
\begin{align}
&\Big\vert \int_{\Omega^{2}}S(x,m,y,n)\psi_{1}(t,x,m)\psi_{1}(t,y,n) \mathrm{s}(t,x)\ \de x\de m \de y \de n \notag\\
&\quad-\int_{\Omega^{2}}S(x,m,y,n)\psi_{2}(t,x,m)\psi_{2}(t,y,n) \mathrm{s}(t,x)\ \de x\de m \de y \de n\Big\vert \notag\\
&\leq \int_{\Omega^{2}}\left\vert S(x,m,y,n)(\psi_{1}-\psi_{2})(t,x,m)\right\vert \left\vert \psi_{1}(t,y,n)  \right\vert \ \de x\de m \de y \de n \notag\\
&\quad+\int_{\Omega^{2}}\left\vert S(x,m,y,n)(\psi_{1}-\psi_{2})(t,y,n)\right\vert \left\vert \psi_{2}(t,x,m) \right\vert\ 
\de x\de m \de y \de n\leq 2\overline{S}_{0}\left\Vert (\psi_{1}-\psi_{2}) (t,\cdot)\right\Vert_{1},
\label{S est}
\end{align}
where we used that $\psi_{1},\psi_{2}$ are probability densities. Hence, we conclude that 
\begin{align}
I_{2}\leq 2\overline{S}_{0}\left\Vert (\psi_{1}-\psi_{2})(t,\cdot)\right\Vert_{1}\lesssim \left\Vert (\psi_{1}-\psi_{2})(t,\cdot)\right\Vert_{1} . \label{est of I2}      
\end{align}
Gathering \eqref{est for I1} and \eqref{est of I2} we get 
\begin{align}
\frac{\de}{\de t}\left\Vert (\mu_{1}-\mu_{2})(t,\cdot)\right\Vert_{1}\lesssim \left\Vert (\mu_{1}-\mu_{2})(t,\cdot)\right\Vert_{1}+\left\Vert (\psi_{1}-\psi_{2})(t,\cdot)\right\Vert_{1}. \label{mu stablity}        
\end{align}
\textbf{Step 2}. \textit{Calculation of $\frac{\de}{\de t}\left\Vert (\psi_{1}-\psi_{2})(t,\cdot)\right\Vert_{1} $}. Put $\mathbf{s}(t,x,m)\coloneqq \mathrm{sgn}(\psi_{1}(t,x,m)-\psi_{2}(t,x,m))$. We compute that 
\begin{align*}
 \frac{\de}{\de t}\left\Vert (\psi_{1}-\psi_{2})(t,\cdot)\right\Vert_{1}=&\int_{\Omega}\partial_{t}(\psi_{1}-\psi_{2})(t,x,m)\mathbf{s}(t,x,m)\ \de x\de m\\
 =&\int_{\Omega}\mathrm{div}_{x}\left(\mathbf{A}[\psi_{2}]\psi_{2}-\mathbf{A}[\psi_{1}]\psi_{1}\right)(t,x,m)\mathbf{s}(t,x,m)\ \de x\de m\\
 &+\int_{\Omega}\partial_{m}\left(\mathbf{S}[\psi_{2}]\psi_{2}-\mathbf{S}[\psi_{1}]\psi_{1}\right)(t,x,m)\mathbf{s}(t,x,m)\ \de x\de m\coloneqq J_{1}+J_{2}.  
\end{align*}
We start with $J_{1}$. First, we have 
\begin{align*}
J_{1}&=\int_{\Omega}\mathrm{div}_{x}\left((\mathbf{A}[\psi_{2}]-\mathbf{A}[\psi_{1}])\psi_{2}\right)\mathbf{s}(t,x,m)\ \de x\de m\\
&\quad+\int_{\Omega}\mathrm{div}_{x}(\left(\psi_{2}-\psi_{1}\right)\mathbf{A}[\psi_{1}])\ \mathbf{s}(t,x,m)\ \de x\de m\coloneqq J_{1}^{1}+J_{1}^{2}. 
\end{align*}
Since $\mathrm{div}_{x}(\mathbf{a})=0$ we have 
\begin{align*}
J_{1}^{1}&=\int_{\Omega}(\mathbf{A}[\psi_{2}]-\mathbf{A}[\psi_{1}])\nabla_{x} \psi_{2}(t,x,m)\mathbf{s}(t,x,m)\ \de x \de m\\
&\leq \left\Vert \nabla_{x} \psi_{2}\right\Vert_{L^{\infty}_{t}L^{1}_{x}}\left\Vert \mathbf{a}\right\Vert_{\infty}\left\Vert (\psi_{1}-\psi_{2})(t,\cdot)\right\Vert_{1}\lesssim \left\Vert (\psi_{1}-\psi_{2})(t,\cdot)\right\Vert_{1}.          
\end{align*}
Furthermore, one has 
\begin{align*}
 J_{1}^{2}&=-\int_{\Omega}\nabla_{x}\left\vert \psi_{1}-\psi_{2}\right\vert(t,x,m)\mathbf{A}[\psi_{1}](t,x)\ \de x\de m\\
 &=\int_{\Omega}\left\vert \psi_{1}-\psi_{2}\right\vert(t,x,m)\mathrm{div}_{x}\mathbf{A}[\psi_{1}](t,x)\  \de x \de m=0.     
\end{align*}
It follows that 
\begin{align}
J_{1}\lesssim \left\Vert (\psi_{1}-\psi_{2})(t,\cdot)\right\Vert_{1}. \label{J1 est}     
\end{align}
We continue by estimating $J_{2}$. By adding and subtracting $\mathbf{S}[\psi_1]\psi_2$ we have
\begin{align*}
J_{2}=&\int_{\Omega}\partial_{m}(\mathbf{S}[\psi_{2}]-\mathbf{S}[\psi_{1}])\psi_{2}(t,x,m)\mathbf{s}(t,x,m)\ \de x\de m\\
&+\int_{\Omega}(\mathbf{S}[\psi_{2}]-\mathbf{S}[\psi_{1}])\partial_{m}\psi_{2}(t,x,m)\mathbf{s}(t,x,m)\ \de x \de m \\
&+\int_{\Omega}\mathbf{S}[\psi_{1}](t,x,m)\partial_{m}\left\vert \psi_{2}-\psi_{1}\right\vert(t,x,m)\ \de x \de m\\  &+\int_{\Omega}\partial_{m}\mathbf{S}[\psi_{1}](t,x,m)\left\vert \psi_{1}-\psi_{2}\right\vert(t,x,m)\ \de x\de m=\sum_{k=1}^{4}J_{2}^{k}.   
\end{align*}
We bound each of the $J_{2}^{k}$ separately. The same argument demonstrated in \eqref{S est} yields    
\begin{align}
J_{2}^{1}\leq 2\overline{S}_{1}\left\Vert (\psi_{2}-\psi_{1})(t,\cdot)\right\Vert_{1}. \label{J21est}     
\end{align}
Moreover, it holds that 
\begin{align}
J_{2}^{2}\leq 2\overline{S}_{0}\left\Vert \partial_{m}\psi_{2}\right\Vert_{L^{\infty}_{t}L^{1}_{x}}\left\Vert (\psi_{1}-\psi_{2})(t,\cdot)\right\Vert_{1}\lesssim \left\Vert (\psi_{1}-\psi_{2})(t,\cdot)\right\Vert_{1}.     \label{J22 est}   
\end{align}
Moreover, integrating by parts reveals that 
\begin{align*}
J_{2}^{4}=-\int_{\Omega}\partial_{m}\mathbf{S}[\psi_{1}]\left\vert \psi_{2}-\psi_{1}\right\vert(t,x,m)\ \de x\de m ,
\end{align*}
and therefore, we get  
\begin{align}
 J_{2}^{3}+J_{2}^{4}=0. \label{J23 J24 vanish}
\end{align}
Gathering \eqref{J21est}-\eqref{J23 J24 vanish} we find that 
\begin{align}
J_{2}\lesssim \left\Vert (\psi_{1}-\psi_{2})(t,\cdot)\right\Vert_{1}. \label{J2est}     
\end{align}
Combining \eqref{J1 est} with \eqref{J2est} we deduce that 
\begin{align}
 \frac{\de}{\de t}\left\Vert (\psi_{1}-\psi_{2})(t,\cdot)\right\Vert_{1}\lesssim \left\Vert (\psi_{1}-\psi_{2})(t,\cdot)\right\Vert_{1}. \label{psi stability}      
\end{align}
The combination of \eqref{mu stablity} and \eqref{psi stability} entails 
\begin{align*}
 \frac{\de}{\de t}&\left(\left\Vert (\psi_{1}-\psi_{2})(t,\cdot)\right\Vert_{1}+\left\Vert (\mu_{1}-\mu_{2})(t,\cdot)\right\Vert_{1}\right)\lesssim \left\Vert (\psi_{1}-\psi_{2})(t,\cdot)\right\Vert_{1}+\left\Vert (\mu_{1}-\mu_{2})(t,\cdot)\right\Vert_{1}.     
\end{align*}
Gr\"onwall's lemma implies the asserted inequality. 
\qed
\begin{rem}
All the estimates in Lemmas \ref{Positivity}-\ref{nablax psi bound} do not depend on the mollification  parameter $R$ and hence they hold for the limit $\psi$ as well. However, the utility of the logarithmic gradient bounds will become appearent only in the next section.     
\end{rem} 

\subsection{Existence of weak solutions for the $N$-particle hierarchy \eqref{Kolmogorov eq}}
Next we prove the existence of a weak solution to  \eqref{Kolmogorov eq sec 2} in the sense of Definition \ref{def of entropy sol}, as described in Proposition \ref{existence for Kolmogorov}. As a preliminary, we recall the dual variational formulation of the entropy. 
Let $\Omega \subset \mathbb{R}^d$ and let $\psi \in \mathscr{P}(\Omega)$. Then, it holds that \begin{align}\label{variational entropy}
\int_{\Omega}\psi(y)\log(\psi(y))\ \de y=\underset{\Phi \in C_{b}(\Omega)}{\sup}\left(\int_{\Omega}\psi(y)\Phi(y)\ \de y-\log \int_{\Omega} \exp(\Phi(y))\ \de y\right),  \end{align}
see \cite{dembo2009large} for instance.We introduce the following shorthand notation. Let $\eta_{R}(m)$ be such that:
\begin{itemize}
    \item $\eta_{R}\in C^{\infty}_{0}([0,\infty))$. 
    \item $\eta_{R}\equiv 1$ on $[0,R]$, $\eta_{R}\equiv 0$ on $(2R,\infty)$ and $\left\Vert \partial_{m}\eta_{R}\right\Vert_{\infty}\leq `1$. 
\end{itemize}
\begin{proof}(\textit{of Proposition \ref{existence for Kolmogorov}}). 
Denote by $\psi_{N,R}\in C^{1}([0,T]\times \Omega_{N})$ the unique solution of the truncated and mollified equation
\begin{align}
\partial_{t}\psi_{N,R}&+\frac{1}{N}\sum_{i=1}^{N}\textnormal{div}_{x_{i}}\bigg(\sum_{j=1}^{N}\eta_{R}(m_{j})\mathbf{a}_{R}(x_{i}-x_{j})\psi_{N,R}\bigg) \notag\\
&+\frac{1}{N}\sum_{i=1}^{N}\partial_{m_{i}}\bigg(\sum_{ j=1}^{N}S_{R}(x_{i},m_{i},x_{j},m_{j})\psi_{N,R}\bigg)=0, \ \psi_{N,R}(0,\cdot)=\psi_{N,0}, \label{truncated eq}  
\end{align}
where $\mathbf{a}_R$ and $S_R$ are the same as in \eqref{mollification properties}. Multiplying the above equation by $e^{\frac{1}{2}\sum_{k=1}^{N}m_{k}}$ and setting $\Phi_{N,R}=e^{\frac{1}{2}\sum_{k=1}^{N}m_{k}}\psi_{N,R}$ we get 
\begin{align*}
\partial_{t}\Phi_{N,R}&+\frac{1}{N}\sum_{i=1}^{N}\mathrm{div}_{x_{i}}\left(\sum_{j=1}^{N}\eta_{R}(m_{j})\mathbf{a}_{R}(x_{i}-x_{j})\Phi_{N,R}\right)\\
&+\frac{1}{N}\sum_{i=1}^{N}\partial_{m_{i}}\left(\sum_{j=1}^{N}S_{R}(x_{i},m_{i},x_{j},m_{j})\Phi_{N,R}\right)-\frac{1}{2N}\sum_{i=1}^{N}\sum_{j=1}^{N}S_{R}(x_{i},m_{i},x_{j},m_{j})\Phi_{N,R}=0. 
\end{align*}
Since $\Phi_{N,R}(t,\cdot)\in L^{1}(\Omega_{N})\cap \mathscr{P}(\Omega_{N})$ it must be that $\Phi_{N,R}$ attains an inner global maximum $(\mathbf{x}_{N}^{t},\mathbf{m}_{N}^{t})$ and therefore using that $\nabla_{\mathbf{m}_{N}
}\Phi_{N,R}=0$ and $\nabla_{\mathbf{x}_{N}}\Phi_{N,R}=0$ we get 
\begin{align*}
\frac{\de}{\de t}\left\Vert \Phi_{N,R}(t,\cdot)\right\Vert_{\infty}&+\frac{1}{N}\sum_{i=1}^{N}\sum_{j=1}^{N}\partial_{m_{i}}S_{R}(x_{i},m_{i},x_{j},m_{j})\left\Vert \Phi_{N,R}(t,\cdot)\right\Vert_{\infty}\\
&-\frac{1}{2N}\sum_{i=1}^{N}\sum_{j=1}^{N}S_{R}(x_{i},m_{i},x_{j},m_{j})\left\Vert \Phi_{N,R}(t,\cdot)\right\Vert_{\infty}=0, 
\end{align*}
which leads to the estimate 
\begin{align*}
\frac{\de}{\de t}\left\Vert \Phi_{N,R}(t,\cdot)\right\Vert_{\infty}\leq N(\overline{S}_{1}+\frac{1}{2}\overline{S}_{0})\left\Vert \Phi_{N,R}(t,\cdot)\right\Vert_{\infty}.
\end{align*}
We therefore deduce the inequality  
\begin{align}
\left\Vert \Phi_{N,R}(t,\cdot)\right\Vert_{\infty}\leq e^{N(\overline{S}_{1}+\frac{1}{2}\overline{S}_{0})T}\left\Vert \psi_{N,0}e^{\frac{1}{2}\sum_{k=1}^{N}m_{k}}\right\Vert_{\infty}\leq e^{N(\overline{S}_{1}+\frac{1}{2}\overline{S}_{0})T}\overline{\mathbf{E}}_{N} \label{est on PhiN,R}.  \end{align}
Given a test function $\varphi \in W^{1,1}(\Omega_{N})$ with $\left\Vert \varphi\right\Vert_{W^{1,1}}\leq 1$ we have
\begin{align}
&\int_{\Omega_{N}}(\psi_{N,R}(t,\mathbf{x}_{N},\mathbf{m}_{N})-\psi_{N,R}(s,\mathbf{x}_{N},\mathbf{m}_{N}))e^{\frac{1}{4}\sum_{k=1}^{N}m_{k}}\varphi(\mathbf{x}_{N},\mathbf{m}_{N})\ \de \mathbf{x}_{N}\de \mathbf{m}_{N}\notag\\=\,&\frac{1}{N}\sum_{i,j}\int_{s}^{t}\int_{\Omega_{N}}\eta_{R}(m_{j})\nabla_{x_{i}}\varphi(\mathbf{x}_{N},\mathbf{m}_{N})\mathbf{a}_{R}(x_{i}-x_{j})\psi_{N,R}(\tau,\mathbf{x}_{N},\mathbf{m}_{N})e^{\frac{1}{4}\sum_{k=1}^{N}m_{k}}\ \de \mathbf{x}_{N}\de\mathbf{m}_{N}\de \tau \notag\\
&+\frac{1}{N}\sum_{i,j}\int_{s}^{t}\int_{\Omega_{N}}S_R(x_{i},m_{i},x_{j},m_{j})\partial_{m_{i}}\varphi(\mathbf{x}_{N},\mathbf{m}_{N})\psi_{N,R}(\tau,\mathbf{x}_{N},\mathbf{m}_{N})e^{\frac{1}{4}\sum_{k=1}^{N}m_{k}}\ \de \mathbf{x}_{N}\de\mathbf{m}_{N}\de \tau \notag\\
&+\frac{1}{4N}\sum_{i,j} \int_{s}^{t}\int_{\Omega_{N}} S_R(x_{i},m_{i},x_{j},m_{j})\varphi(\mathbf{x}_{N},\mathbf{m}_{N})\psi_{N,R}(\tau,\mathbf{x}_{N},\mathbf{m}_{N})e^{\frac{1}{4}\sum_{k=1}^{N}m_{k}} \de \mathbf{x}_{N}\de\mathbf{m}_{N}\de \tau. \label{weak formulation applied for phi}
\end{align}
In what follows $\lesssim$ designates an inequality up to a constant independent of $R$. In view of \eqref{est on PhiN,R}, we obtain the following estimates. First, since $\eta_{R}(m_{j})e^{\frac{1}{4}\sum_{k=1}^{N}m_{k}}\lesssim e^{\frac{1}{2}\sum_{k=1}^{N}m_{k}}$, we have   
\begin{align*}
\int_{\Omega_{N}}\eta_{R}(m_{j})\nabla_{x_{i}}\varphi(\mathbf{x}_{N},\mathbf{m}_{N})\mathbf{a}_{R}(x_{i}-x_{j})\psi_{N,R}(\tau,\mathbf{x}_{N},\mathbf{m}_{N})e^{\frac{1}{4}\sum_{k=1}^{N}m_{k}}\ \dd \mathbf{x}_{N}\dd \mathbf{m}_{N}\\
\leq \left\Vert \mathbf{a}\right\Vert_{\infty} \left\Vert \Phi_{N,R}\right\Vert_{L^{\infty}_{t}L^{\infty}_{x}}\left\Vert \nabla_{\mathbf{x}_{N}}\varphi\right\Vert_{L^{1}}\lesssim 1 .      
\end{align*}
In addition, we have 
\begin{align*}
\int_{\Omega_{N}}S_{R}(x_{i},m_{i},x_{i},x_{j})\partial_{m_{i}}&\varphi(\mathbf{x}_{N},\mathbf{m}_{N})\psi_{N,R}(\tau,\mathbf{x}_{N},\mathbf{m}_{N})e^{\frac{1}{4}\sum_{k=1}^{N}m_{k}}\ \dd \mathbf{x}_{N}\dd \mathbf{m}_{N}\\
&\leq \overline{S}_{0}\left\Vert \nabla_{\mathbf{m}_{N}} \varphi\right\Vert_{1}\left\Vert \Phi_{N,R}\right\Vert_{\infty}\lesssim 1 
\end{align*}
and 
\begin{align*}
\int_{\Omega_{N}}&S_{R  }(x_{i},m_{i},x_{j},m_{j})\varphi(\mathbf{x}_{N},\mathbf{m}_{N})\psi_{N,R}(\tau,\mathbf{x}_{N},\mathbf{m}_{N})e^{\frac{1}{4}\sum_{k=1}^{N}m_{k}}\ \dd \mathbf{x}_{N}\dd\mathbf{m}_{N}\\
&\leq \overline{S}_{0}\left\Vert \varphi\right\Vert_{1}\left\Vert \Phi_{N,R}\right\Vert_{L^{\infty}_{t}L^{\infty}_\infty}\lesssim 1.       
\end{align*}
Owing to \eqref{weak formulation applied for phi} we conclude that  
\begin{align*}
\left\Vert e^{\frac{1}{4}\sum_{k=1}^{N}m_{k}}(\psi_{N,R}(t,\cdot)-\psi_{N,R}(s,\cdot))\right\Vert_{W^{-1,1}}\leq  C\left\vert t-s\right\vert
\end{align*}
where $C>0$ is independent of $R$ and $W^{-1,1}(\Omega_{N})$ denotes the dual of $W^{1,1}(\Omega_{N})$. As a result, it follows from the theorem of Arzel\'a-Ascoli that there is a sequence $R_{n}$ and some $\Psi_{N}\in C\left([0,T];W^{-1,1}(\Omega_{N})\right)$ such that 
\begin{align}
\underset{t\in [0,T]}{\sup}\left\Vert e^{\frac{1}{4}{\sum_{k=1}^{N}m_{k}}}\psi_{N,R_{n}}(t,\cdot)-\Psi_{N}(t,\cdot)\right\Vert_{W^{-1,1}} \underset{n\rightarrow \infty}{\rightarrow} 0.\label{weighted W1 convergence}     
\end{align}
In addition, defining $\psi_{N}=e^{-\frac{1}{4}{\sum_{k=1}^{N}m_{k}}}\Psi_{N}$,  \eqref{weighted W1 convergence} clearly also implies  
\begin{align*}
\underset{t \in [0,T]}{\sup}\left\Vert \psi_{N,R_{n}}(t,\cdot)-\psi_{N}(t,\cdot)\right\Vert _{W^{-1,1}}\underset{n\rightarrow \infty}{\rightarrow} 0.    
\end{align*}
Furthermore, the uniform in $R$ bound in \eqref{est on PhiN,R}
entails in particular that 
\begin{align*}
\left\Vert e^{\frac{1}
{2}\sum_{k=1}^{N}m_{k}}\psi_{N}(t,\cdot)\right\Vert_{\infty}\lesssim 1.     
\end{align*}
Thus we are able to pass to the limit as $n\rightarrow \infty$ in the weak formulation 
and infer that $\psi_{N}$ is the requested weak solution. 

In order to prove that the accumulation point $\psi_{N}$ satisfies the asserted entropy inequality \eqref{entropy} we compute the time derivative of the  entropy of $\psi_{N,R}$ (for brevity we omit time dependency). 
\begin{align}
\frac{\de}{\de t}\frac{1}{N}\int_{\Omega_{N}}\psi_{N,R_{n}}&(\mathbf{x}_{N},\mathbf{m}_{N})\log\left(\psi_{N,R_{n}}(\mathbf{x}_{N},\mathbf{m}_{N})\right)\ \de\mathbf{x}_{N}\de\mathbf{m}_{N}\notag\\
=&\,\frac{1}{N}\int_{\Omega_{N}}\partial_{t}\psi_{N,R_{n}}(\mathbf{x}_{N},\mathbf{m}_{N})\log\left(\psi_{N,R_{n}}(\mathbf{x}_{N},\mathbf{m}_{N})\right)\ \de\mathbf{x}_{N}\de\mathbf{m}_{N}\notag\\
=&\frac{1}{N^{2}}\sum_{i,j} \int_{\Omega_{N}} S_{R_{n}}(x_{i},m_{i},x_{j},m_{j})\partial_{m_{i}}\psi_{N,R_{n}}(\mathbf{x}_{N},\mathbf{m}_{N})\ \de\mathbf{x}_{N}\de\mathbf{m}_{N}\notag\\
&+\frac{1}{N^{2}}\sum_{i,j} \int_{\Omega_{N}} \eta_{R_{n}}(m_{j})\mathbf{a}_{R_{n}}(x_{i}-x_{j})\nabla_{x_{i}}\psi_{N,R_{n}}(\mathbf{x}_{N},\mathbf{m}_{N})\ \de \mathbf{x}_{N}\de\mathbf{m}_{N}
\notag\\=&-\frac{1}{N^{2}}\sum_{i,j} \int_{\Omega_{N}} \partial_{m_{i}}S_{R_{n}}(x_{i},m_{i},x_{j},m_{j})\psi_{N,R_{n}}(\mathbf{x}_{N},\mathbf{m}_{N})\ \de\mathbf{x}_{N}\de\mathbf{m}_{N}. 
\label{time derivative of entropy of psi}
\end{align}
Integrating in time implies that 
\begin{align}
\frac{1}{N}\int_{\Omega_{N}} \psi_{N,R_{n}}&(t,\mathbf{x}_{N},\mathbf{m}_{N})\log(\psi_{N,R_{n}}(t,\mathbf{x}_{N},\mathbf{m}_{N}))\ \de\mathbf{x}_{N}\de\mathbf{m}_{N}
\notag\\&=
\frac{1}{N}\int_{\Omega_{N}} \psi_{N,0}(\mathbf{x}_{N},\mathbf{m}_{N})\log(\psi_{N,0}(\mathbf{x}_{N},\mathbf{m}_{N}))\ \de\mathbf{x}_{N}\de\mathbf{m}_{N} \notag\\
&\quad-\frac{1}{N^2}\underset{i,j}{\sum}\int_{0}^{t}\int_{\Omega_{N}} \partial_{m_{i}}S_{R_{n}}(x_{i},m_{i},x_{j},m_{j})\psi_{N,R_{n}}(\tau,\mathbf{x}_{N},\mathbf{m}_{N})\ \de\mathbf{x}_{N}\de\mathbf{m}_{N}\de\tau. 
\label{integration in time}
\end{align}
Note that for each fixed $i,j$ we have that $e^{-\frac{1}{4}\sum_{k=1}^{N}m_{k}}\partial_{m_{i}}S(x_{i},m_{i},x_{j},m_{j})\in W^{1,1}(\Omega_{N})$ and so 
\begin{align*}
&\left\vert \int_{\Omega_{N}}\partial_{m_{i}}S(x_{i},x_{j},m_{i},m_{j})(\psi_{N,R_{n}}-\psi_{N})(\tau,\mathbf{x}_{N},\mathbf{m}_{N}
)\ \dd \mathbf{x}_{N}\dd \mathbf{m}_{N}\right\vert\\
&=\left\vert \int_{\Omega_{N}}e^{-\frac{1}{4}\sum_{k=1}^{N}m_{k}}\partial_{m_{i}}S(x_{i},m_{i},x_{j},m_{j})(\psi_{N,R_{n}}-\psi_{N})(\tau,\mathbf{x}_{N},\mathbf{m}_{N}
)e^{\frac{1}{4}\sum_{k=1}^{N}m_{k}}\ \dd \mathbf{x}_{N}\dd \mathbf{m}_{N}\right\vert\underset{n\rightarrow \infty}{\rightarrow}0.   \end{align*}
Consequently, it follows that 
\begin{align}
&\vert \int_{\Omega_{N}}\partial_{m_{i}}S_{R_{n}}(x_{i},m_{i},x_{j},m_{j})\psi_{N,R_{n}}(\tau,\mathbf{x}_{N},\mathbf{m}_{N})\ \dd \mathbf{x}_{N}\dd \mathbf{m}_{N} \notag\\
&-\int_{\Omega_{N}}\partial_{m_{i}}S(x_{i},m_{i},x_{j},m_{j})\psi_{N}(\tau,\mathbf{x}_{N},\mathbf{m}_{N})\ \dd\mathbf{x}_{N}\dd \mathbf{m}_{N} \vert\notag\\
&\leq \left\Vert \partial_{m_{i}}S_{R_{n}}-\partial_{m_{i}}S\right\Vert_{\infty}+\left\vert \int_{\Omega_{N}}\partial_{m_{i}}S(x_{i},m_{i},x_{j},m_{j})(\psi_{N,R_{n}}-\psi_{N})(\tau,\mathbf{x}_{N},\mathbf{m}_{N})\ \dd \mathbf{x}_{N}\dd \mathbf{m}_{N}\right\vert\underset{n\rightarrow \infty}{\rightarrow} 0.  \label{est on S against psi} 
\end{align}
Combining  \eqref{integration in time}-\eqref{est on S against psi} we arrive at 
\begin{align*}
&\underset{n\rightarrow \infty}{\lim}\frac{1}{N}\int_{\Omega_{N}}\psi_{N,R_{n}}(t,\mathbf{x}_{N},\mathbf{m}_{N})\log(\psi_{N,R_{n}}(\mathbf{x}_{N},\mathbf{m}_{N}))\ \dd \mathbf{x}_{N}\dd \mathbf{m}_{N}\\
&= \mathcal{H}_{N}(0)-\frac{1}{N^{2}}\sum_{i,j}\int_{0}^{t}\int_{\Omega_{N}}\partial_{m_{i}}S(x_{i},m_{i},x_{j},m_{j})\psi_{N}(\tau,\mathbf{x}_{N},\mathbf{m}_{N})\ \dd \mathbf{x}_{N}\dd\mathbf{m}_N.     
\end{align*}
To conclude we use that the relative entropy is a lower semi-continuous functional, a fact whose proof we recall. Fix a test function $\Phi \in C_{b}(\Omega_{N})$. Then, by the variational formulation of the relative entropy in \eqref{variational entropy}, we have 
\begin{align*}
\underset{n \rightarrow \infty}{\limsup} &\int_{\Omega_{N}}\psi_{N,R_{n}}(t,\mathbf{x}_{N},\mathbf{m}_{N})\log\left(\psi_{N,R_{n}}(t,\mathbf{x}_{N},\mathbf{m}_{N})\right)\ \de\mathbf{x}_{N}\de\mathbf{m}_{N}\\
\geq&\,\underset{n \rightarrow \infty}{\limsup} \int_{\Omega_{N}}\psi_{N,R_{n}}(t,\mathbf{x}_{N},\mathbf{m}_{N})\Phi(\mathbf{x}_{N},\mathbf{m}_{N})\ \de\mathbf{x}_{N}\de\mathbf{m}_{N}\\
&-\log\int_{\Omega_{N}}\exp\left(\Phi(\mathbf{x}_{N},\mathbf{m}_{N})\right)\ \de\mathbf{x}_{N}\de\mathbf{m}_{N}\\
=&\int_{\Omega_{N}}\psi_{N}(t,\mathbf{x}_{N},\mathbf{m}_{N})\Phi(\mathbf{x}_{N},\mathbf{m}_{N})\ \de\mathbf{x}_{N}\de\mathbf{m}_{N}-\log\int_{\Omega_{N}}\exp\left(\Phi(\mathbf{x}_{N},\mathbf{m}_{N})\right)\ \de\mathbf{x}_{N}\de\mathbf{m}_{N}.
\end{align*}
By taking the supremum over $\Phi \in C_{b}(\Omega_{N})$ in the previous inequality and using the dual variational formulation of the entropy in \eqref{variational entropy}, we conclude \eqref{entropy}. 
\end{proof}

\section{The mean field limit}\label{Sec mean field}
In this part, we introduce the relative entropy (as defined in \eqref{defn_of_entropy}) and prove rigorously a propagation of chaos property via the  relative entropy. The proof is split into 3 steps, summarized below: 
\begin{itemize}
    \item[(I)] We study the evolution of $\mathcal{H}_{N}(t)$ and obtain an inequality of the form 
    \begin{align}
    \mathcal{H}_{N}(t)\leq \mathcal{H}_{N}(0)-\frac{1}{N}\int_{0}^{t}\int_{\Omega_{N}}\psi_{N}(\mathcal{R}_{N}+\mathcal{S}_{N})(\tau,\mathbf{x}_{N},\mathbf{m}_{N})\ \dd \mathbf{x}_{N}\dd \mathbf{m}_{N}\dd \tau. \label{generalstrategy1}
    \end{align}
    This inequality relies crucially on the fact that the weak solution we constructed in the previous section is an entropy solution. 
    \item[(II)] We prove a "cancellation Lemma" which ultimately shows that the quantity 
    \begin{align}
    \frac{1}{N}\int_{\Omega_{N}}\exp(\left\vert \mathcal{R}_{N}+\mathcal{S}_{N}\right\vert)\overline{\psi_{N}}(t,\mathbf{x}_{N},\mathbf{m}_{N}) \ \dd \mathbf{x}_{N}\dd \mathbf
    {m}_{N} \label{general strategy 2} 
    \end{align}
is of order $\frac{1}{N}$. This is where the smallness of the parameters $\left\Vert \mathbf{a}\right\Vert_{\infty},\overline{S}_{0},\overline{S}_{1},T_{\ast\ast}$ comes into play. 
\item[(III)] We apply a "change of laws lemma"  in order to control the integral in \eqref{generalstrategy1} by means of \eqref{general strategy 2} and the relative entropy. This eventually leads to an evolution inequality of the form 
     \begin{align*}
     \mathcal{H}_{N}(t)\leq \mathcal{H}_{N}(0)+\int_{0}^{t}\mathcal{H}_{N}(\tau) \ \dd \tau+O(\frac{1}{N}),    
     \end{align*}
which concludes the proof. 
\end{itemize}
We proceed by proving (I), (III) and finally (II). 
\subsection{Evolution in time of the relative entropy} \label{sec:POC}

We study the evolution of the relative entropy between the $N$-particle distribution function $\psi_N$ and the factorized PDE solution $\overline{\psi_{N}}:= \psi^{\otimes N}$. 
\begin{prop}\label{prop evolution of entropy}
Let assumptions \textbf{H1-H3} hold. Let $\psi_{N}$ be a weak solution to \eqref{Kolmogorov eq sec 2} fulfilling the entropy inequality \eqref{entropy}, as guaranteed via Proposition \ref{existence for Kolmogorov}. Let $\psi$ be the solution to \eqref{limit PDE Intro} with initial data  $\psi_{0}$ (guaranteed via Theorem \ref{well posedness intro}). Then, we have 
\[
\mathcal{H}_{N}(t)\leq \mathcal{H}_{N}(0)-\frac{1}{N}\int_{0}^{t}\int_{\Omega_{N}}\psi_{N}(\tau,\mathbf{x}_{N},\mathbf{m}_{N})(\mathcal{R}_{N}+\mathcal{S}_{N})\ \dd \mathbf{x}_{N}\dd \mathbf{m}_{N}\dd \tau,
\]
with $\mathcal{R}_N$ and $\mathcal{S}_N$ defined in \eqref{definition_R_N} and \eqref{S reminder}. 
\end{prop}

\begin{proof}
We recall that $\psi_{N}$ solves \eqref{Kolmogorov eq sec 2} and $\overline{\psi_{N}}= \psi^{\otimes N}$ solves \eqref{Kolmogorov product}. To make the equations lighter we omit dependency on $(t,\mathbf{x}_{N},\mathbf{m}_{N})$ whenever there is no ambiguity. Note that the  relative entropy between $\psi_{N}$ and $\overline{\psi_{N}}$ can be written as 

\[
\mathcal{H}_{N}(t)=\frac{1}{N}\int_{\Omega_{N}}\psi_{N}\log\left(\psi_{N}\right)\ \de\mathbf{x}_{N}\de \mathbf{m}_{N}-\frac{1}{N}\int_{\Omega_{N}}\psi_{N}\log(\overline{\psi_{N}})\ \de\mathbf{x}_{N}\de\mathbf{m}_{N}.
\]
By Proposition \ref{existence for Kolmogorov} we have 
\begin{align}
\frac{1}{N}&\int_{\Omega_{N}}\psi_{N}\log(\psi_{N})\ \dd \mathbf{x}_{N}\dd \mathbf{m}_{N}\leq \frac{1}{N}\int_{\Omega_{N}}\psi_{N,0}\log(\psi_{N,0})\ \dd \mathbf{x}_{N}\dd\mathbf{m}_{N} \notag\\
&-\frac{1}{N^{2}}\sum_{i,j}\int_{0}^{t}\int_{\Omega_{N}}\partial_{m_{i}}S(x_{i},m_{i},x_{j},m_{j})\psi_{N}(\tau,\mathbf{x}_{N},\mathbf{m}_{N})\ \dd\mathbf{x}_{N}\dd \mathbf{m}_{N}\dd \tau. \label{first part of relative entropy}     
\end{align}
Note that  $\log(\overline{\psi_{N}})$ is an admissible test function: 
indeed by Lemma \ref{nablax psi bound} we have 
\begin{align*}
&\frac{1}{N}\sum_{i,j}\int_{\Omega_{N}}m_{j}\mathbf{a}(x_{i}-x_{j})\cdot \nabla_{x_{i}}\log(\psi(t,x_{i},m_{i}))\psi_{N}(t,\mathbf{x}_{N},\mathbf{m}_{N})\ \dd \mathbf{x}_{N}\dd \mathbf{m}_{N}\\
&\leq \left\Vert \mathbf{a}\right\Vert_{\infty}\left\Vert \nabla_{x}\log(\psi)\right\Vert_{L^{\infty}_{t}L^{\infty}
_{x,m}}\frac{1}{N}\sum_{i,j}\int_{\Omega_{N}}m_{j}\psi_{N}(t,\mathbf{x}_{N},\mathbf{m}_{N})\ \dd\mathbf{x}_{N}\dd\mathbf{m}_{N} \lesssim 1,      
\end{align*}
where we used propagation of the weight moments for $\psi_{N}$ in the last inequality, which is a particular consequence of \eqref{exp moments est}.
Furthermore, by Lemma \ref{log grad est weight} we have  
\begin{align*}
\frac{1}{N}\int_{\Omega_{N}}\sum_{i,j}&S(x_{i},m_{i},x_{j},m_{j})\psi_{N}(t,\mathbf{x}_{N},\mathbf{m}_{N})\partial_{m_{i}}\log(\psi(t,x_{i},m_{i}))\ \dd \mathbf{x}_{N}\dd \mathbf{m}_{N} \\
&\leq N\overline{S}_{0}\left\Vert \partial_{m}\log(\psi)\right\Vert_{L^{\infty}_{t}L^{\infty}_{x,m}}\lesssim 1.    
\end{align*}
In addition, by Lemma \ref{log grad est weight} there is a constant $A>0$ such that  
\begin{align*}
-m-A\leq \log(\psi)\leq -m+A\leq m+A \Rightarrow \vert\log(\psi)\vert\leq m+A.      
\end{align*}
Therefore, we have 
\begin{align*}
\int_{\Omega_{N}}\left\vert \log(\overline{\psi_{N}})\psi_{N}\right\vert \ \dd\mathbf{x}_{N}\dd \mathbf{m}_{N}
&\leq \sum_{k=1}^{N}\int_{\Omega_{N}}\psi_{N}(m_{k}+A)\ \dd \mathbf{x}_{N}\dd\mathbf{m}_{N}\\
&\lesssim \int_{\Omega_{N}}\psi_{N}e^{\sum_{k=1}^{N}\frac{m_{k}}{2}}\ \dd \mathbf{x}_{N}\dd \mathbf{m}_{N}\lesssim 1.      
\end{align*}
So $\log(\overline{\psi_{N}})$ is an admissible test function. Invoking the weak formulation,  and noticing that $$\nabla_{x_{i}}\log(\overline{\psi_{N}}(\mathbf{x}_{N},\mathbf{m}_{N}))=\nabla_{x_{i}}\log(\psi(x_{i},m_{i})),\ \partial_{m_{i}}\log(\overline{\psi_{N}}(\mathbf{x}_{N},\mathbf{m}_{N}))=\partial_{m_{i}}\log(\psi(x_{i},m_{i}))$$ 
we get 
\begin{align}
&\frac{1}{N}\int_{\Omega_{N}}\log(\overline{\psi_{N}})\psi_{N}\ \dd \mathbf{x}_{N}\dd \mathbf{m}_{N} \notag\\
&=\frac{1}{N}\int_{\Omega_{N}}\log(\overline{\psi_{N,0}})\psi_{N,0}\ \dd \mathbf{x}_{N}\dd\mathbf{m}_{N}+\frac{1}{N}\int_{0}^{t}\int_{\Omega_{N}}\partial_{\tau}\log(\overline{\psi_{N}})\psi_{N}\ \dd \mathbf{x}_{N}\dd \mathbf{m}_{N}\dd \tau \notag\\
&+\frac{1}{N^{2}}\sum_{i,j}\int_{0}^{t}\int_{\Omega_{N}}m_{j}\mathbf{a}(x_{i}-x_{j})\nabla_{x_{i}}\log(\psi(x_{i},m_{i}))\psi_{N}\ \dd\mathbf{x}_{N}\dd\mathbf{m}_{N}\dd\tau \notag\\
&+\frac{1}{N^{2}}\int_{0}^{t}\int_{\Omega_{N}}\sum_{i,j}S(x_{i},m_{i},x_{j},m_{j})\psi_{N}\partial_{m_{i}}\log(\psi(x_{i},m_{i}))\ \dd \mathbf{x}_{N}\dd \mathbf{m}_{N}\dd\tau. \label{second part of relative entropy}
\end{align}
Gathering \eqref{first part of relative entropy}-\eqref{second part of relative entropy} we obtain 
\begin{align}
\mathcal{H}_{N}(t)&\leq \mathcal{H}_{N}(0)-\frac{1}{N}\int_{0}^{t}\int_{\Omega_{N}}\partial_{\tau}\log(\overline{\psi_{N}})\psi_{N}\ \dd \mathbf{x}_{N}\dd\mathbf{m}_{N}\dd \tau \notag\\
&-\frac{1}{N^{2}}\sum_{i,j}\int_{0}^{t}\int_{\Omega_{N}}\partial_{m_{i}}S(x_{i},m_{i},x_{j},m_{j})\psi_{N}\ \dd \mathbf{x}_{N}\dd \mathbf{m}_{N}\dd \tau \notag\\
&-\frac{1}{N^{2}}\sum_{i,j}\int_{0}^{t}\int_{\Omega_{N}}m_{j}\mathbf{a}(x_{i}-x_{j})\cdot\nabla_{x_{i}} \log(\psi(x_{i},m_{i}))\psi_{N}\ \dd \mathbf{x}_{N}\dd\mathbf{m}_{N}\
\notag\\
&-\frac{1}{N^{2}}\sum_{i,j}\int_{0}^{t}\int_{\Omega_{N}}S(x_{i},m_{i},x_{j},m_{j})\partial_{m_{i}}\log(\psi(x_{i},m_{i}))\psi_{N}\ \dd \mathbf{x}_{N}\dd \mathbf{m}_{N}. \label{prefinal entropy ine}
\end{align}
By Lemma \ref{eq for overlinepsiN}, we compute that 
\begin{align}
\frac{1}{N}\partial_{t}\log(\overline{\psi_{N}})=&-\frac{1}{N^{2}}\sum_{i,j}m_{j}\mathbf{a}(x_{i}-x_{j})\cdot\frac{\nabla_{x_{i}}\overline{\psi_{N}}}{\overline{\psi_{N}}} \notag\\
&-\frac{1}{N^{2}}\sum_{i,j}\partial_{m_{i}}S(x_{i},m_{i},x_{j},m_{j})-\frac{1}{N^{2}}\sum_{i,j}S(x_{i},m_{i},x_{j},m_{j})\frac{\partial_{m_{i}}\overline{\psi_{N}}}{\overline{\psi_{N}}}
+\mathcal{R}_{N}+\mathcal{S}_{N} \notag\\
=&-\frac{1}{N^{2}}\sum_{i,j}m_{j}\mathbf{a}(x_{i}-x_{j})\cdot\nabla_{x_{i}}\log(\psi(x_{i},m_{i}))-\frac{1}{N^{2}}\sum_{i,j}\partial_{m_{i}}S(x_{i},m_{i},x_{j},m_{j}) \notag\\
&-\frac{1}{N^{2}}\sum_{i,j}S(x_{i},m_{i},x_{j},m_{j})\partial_{m_{i}}\log(\psi(x_{i},m_{i}))+\mathcal{R}_{N}+\mathcal{S}_{N}.\label{calculation of partial log psiN}
\end{align}
Substituting \eqref{calculation of partial log psiN} inside \eqref{prefinal entropy ine} yields 
\begin{align*}
\mathcal{H}_{N}(t)\leq \mathcal{H}_{N}(0)-\frac{1}{N}\int_{0}^{t}\int_{\Omega_{N}}\psi_{N}(\tau,\mathbf{x}_{N},\mathbf{m}_{N})(\mathcal{R}_{N}+\mathcal{S}_{N})\ \dd\mathbf{x}_{N}\dd\mathbf{m}_{N}\dd\tau,    
\end{align*}
which is the announced inequality. 

\end{proof}

\subsection{Proof of Theorem \ref{main thm}}
We take advantage of the following Lemma, which enables one to replace $\psi_{N}$ with $\overline{\psi_{N}}$. 
\begin{lem}
\label{change rule }\textup{\cite[Lemma 1]{jabin2018quantitative}} Let $\psi_{N},\overline{\psi_{N}}\in\mathscr{P}(\Omega_{N})$
and let $\Phi:\Omega_{N}\rightarrow\mathbb{R}$
be a measurable function such that $e^{N\Phi}\overline{\psi_{N}}\in L^{1}(\Omega_{N})$.
Then, it holds that 

\[
\int_{\Omega_{N}}\Phi\psi_{N}(\mathbf{x}_{N},\mathbf{m}_{N})\ \de\mathbf{x}_{N}\de \mathbf{m}_{N}\leq\mathcal{H}_N(\psi_{N}\left|\overline{\psi_{N}})\right.+\frac{1}{N}\log\int_{\Omega_{N}}e^{N\Phi}\overline{\psi_{N}}(\mathbf{x}_{N},\mathbf{m}_{N})\ \de\mathbf{x}_{N}\de \mathbf{m}_{N}.
\]
\end{lem}
With Lemma \ref{change rule } and the CKP-inequality \eqref{CPK inequality sec pre} we are in a good
position to close the estimate for $\mathcal{H}_N(\psi_{N}\left|\overline{\psi_{N}}\right.)$ stated in our main Theorem \ref{main thm}. A key component of the proof is the following Theorem. 
\begin{thm}
\label{final exp reminder}Let assumptions \textbf{H1}-\textbf{H2} hold and suppose that $\overline{S}_{0},\overline{S}_{1},\left\Vert \mathbf{a}\right\Vert_{\infty},T_{\ast\ast}>0$ are small enough. Let $\psi$ be the unique solution on $[0,T_{\ast\ast}]$ provided by Theorem \ref{well posedness intro}.  Then, there is some $\delta>0$ (independent of $N$) such that for all $t\in [0,T_{\ast\ast}]$ it holds that 
\begin{align*}
\int_{\Omega_{N}}\exp\left(\left\vert \mathcal{R}_{N}+\mathcal{S}_{N}\right\vert \right)\overline{\psi_{N}}(t,\mathbf{x}_{N},\mathbf{m}_{N})\ \de\mathbf{x}_{N}\de\mathbf{m}_{N}\leq\delta.     
\end{align*}
\end{thm}
The proof of Theorem \ref{final exp reminder} requires a careful combinatorial analysis, and we postpone it to the final subsection \ref{section_cancellation}. 
\begin{proof}
\textit{(of Theorem \ref{main thm})}. By Proposition \ref{prop evolution of entropy} we deduce 
\begin{align}
\mathcal{H}_{N}(t)\leq&\,\mathcal{H}_{N}(0)-\frac{1}{N}\int_{0}^{t}\int_{\Omega_{N}}\psi_{N}(\tau,\mathbf{x}_{N},\mathbf{m}_{N})(\mathcal{R}_{N}+\mathcal{S}_{N})\ \de\mathbf{x}_{N}\de\mathbf{m}_{N}\de \tau. \label{reminder of entropy est}  \end{align}
Estimating the inner integral in the right-hand side of \eqref{reminder of entropy est} we obtain
\begin{align}
-\frac{1}{N}&\int_{\Omega_{N}}\psi_{N}(\tau,\mathbf{x}_{N},\mathbf{m}_{N})(\mathcal{R}_{N}+\mathcal{S}_{N})\ \de\mathbf{x}_{N} \de\mathbf{m}_{N} \notag \\
&\leq \frac{1}{N}\int_{\Omega_{N}}\psi_{N}(\tau,\mathbf{x}_{N},\mathbf{m}_{N})\left\vert \mathcal{R}_{N}+\mathcal{S}_{N}\right\vert \ \de\mathbf{x}_{N}\de\mathbf{m}_{N} \notag\\ 
&\leq\mathcal{H}_{N}(\tau)+\frac{1}{N}\log\int_{\Omega_{N}}\exp\left(\left|\mathcal{R}_{N}+\mathcal{S}_{N}\right|\right)\overline{\psi_{N}}(\tau,\mathbf{x}_{N},\mathbf{m}_{N})\ \de \mathbf{x}_{N} \de\mathbf{m}_{N},\label{eq:-48}
\end{align}
where Lemma \ref{change rule } was used with $\Phi=\frac{1}{N}\left\vert \mathcal{R}_{N}+\mathcal{S}_{N}\right\vert$ in the last inequality. By Theorem \ref{final exp reminder} there is some constant $\delta>0$ such that 
\begin{align}
\int_{\Omega_{N}}\exp\left(\left\vert \mathcal{R}_{N}+\mathcal{S}_{N}\right\vert \right)\overline{\psi_{N}}(\tau,\mathbf{x}_{N},\mathbf{m}_{N})\ \de\mathbf{x}_{N}\de\mathbf{m}_{N}\leq\delta \ \mbox{for all}\ \tau\in [0,T_{\ast\ast}], \label{ineq for exp reminder}   
\end{align}
so that in view of \eqref{reminder of entropy est}-\eqref{ineq for exp reminder} we find 
\begin{align*}
\mathcal{H}_{N}(t)\leq \mathcal{H}_{N}(0)+\int_{0}^{t}\mathcal{H}_{N}(\tau)\ \de\tau+\frac{t}{N}\log\delta.     
\end{align*}
By Gr\"onwall's Lemma the latter inequality entails that 
\begin{align}
\mathcal{H}_{N}(t)\leq\left(\mathcal{H}_{N}(0)+\frac{T_{\ast\ast}}{N}\log\delta\right)e^{t}.\label{eq:-49}    
\end{align}
Finally, applying the CKP inequality \eqref{CPK inequality sec pre} and \eqref{eq:-49} we obtain 
\begin{align*}
\left\Vert \psi_{N:k}(t,\cdot)-\psi^{\otimes k}(t,\cdot)\right\Vert_{1}\leq\sqrt{2k\left(\mathcal{H}_{N}(0)+\frac{T_{\ast\ast}}{N}\log\delta\right)e^{t}},    
\end{align*}
so that  
\begin{align*}
\underset{t\in[0,T_{\ast\ast}]}{\sup}\left\Vert \psi_{N:k}(t,\cdot)-\psi^{\otimes k}(t,\cdot)\right\Vert_{1} \underset{N\rightarrow\infty}{\longrightarrow}0 \end{align*}
provided that $\mathcal{H}_{N}(0)\underset{N\rightarrow \infty}{\rightarrow} 0$. 
\end{proof}

\subsection{Proof of Theorem \ref{final exp reminder}} \label{section_cancellation}
We start by recalling some  notation from \cite{jabin2018quantitative}. For each
$p\in\mathbb{N}$ set 

\[
I_{p}\coloneqq(i_{1},\ldots,i_{p})\in \mathbb{N}^{p},\ J_{p}\coloneqq(j_{1},\ldots,j_{p})\in \mathbb{N}^{p}
\]
and given $q\in \mathbb{N}$ define 
\[
\mathcal{T}_{q,p}\coloneqq\left\{ I_{p}\left|1\leq i_{\nu}\leq q \mbox{ for all } 1\leq\nu\leq p\right.\right\} .
\]
Given $I_{p}\in \mathcal{T}_{q,p}$ and $1\leq l \leq q$ we define the \textbf{$l$-th multiplicity} by $a_{l}(I_{p})\coloneqq\left|\left\{ 1\leq\nu\leq p\left|i_{\nu}=l\right.\right\} \right|$.
For each index $I_{p}\in \mathcal{T}_{q,p}$ we define the following quantities 
\begin{align*}
m_{I_{p}}\coloneqq \left\vert \{l ~\vert~ a_{l}(I_{p})=1\}\right\vert, \ n_{I_{p}}\coloneqq\left\vert \{l~\vert ~ a_{l}(I_{p})>1\}\right\vert.       
\end{align*}
The sum $m_{I_{p}}+n_{I_{p}}$ is exactly the number of positive integers  $\leq q$ which appear in $I_{p}$. When there is no ambiguity we write $a_{l}$ instead of $a_{l}(I_{p})$. With these notations we have the following definition. 
\begin{defn}\label{reducedsetdef}
For $k \in \mathbb{N}$ we define the \textbf{reduced set} $\mathscr{R}_{N,2k}$ as the set of all $I_{2k}\in \mathcal{T}_{N,2k}$ such that $1\leq a_{1}\leq a_{2}\leq \dots \leq a_{n_{I_{2k}}+m_{I_{2k}}}$ which implies $a_{n_{I_{2k}}+m_{I_{2k}}+1}=\dots =a_{N}=0$. 
\end{defn}
\begin{defn}\label{defnJ}
Given $(m,n)\in \mathbb{N}^{2}$ we define $\mathcal{J}_{m,n}$ as the set of all $J_{2k}\in \mathcal{T}_{N,2k}$ with multiplicities $(b_{1},\dots,b_{N})$ such that 
\begin{itemize}
  \item[(i)] $b_{l}\geq 1$ for all $1\leq l\leq m$, and
\item[(ii)]  $b_{l}\neq 1$ for all 
$l>m+n$. 
\end{itemize}

\end{defn}
The following law of large numbers type theorem reflects a key component in the analysis performed in \cite{jabin2018quantitative}, as its proof necessitates delicate combinatorial arguments. In our present work, we will apply it after carefully checking that we meet all the necessary conditions, see Lemma \ref{cancellation rule} (Cancellation Lemma) below.
\begin{thm}[\cite{jabin2018quantitative}, Theorem 4]\label{Theorem_jabin_cancel}
Let $\psi \in \mathscr{P}(\Omega)\cap L^{1}(\Omega)$. Let $\chi:\Omega^{2}\rightarrow \mathbb{R}$ be a function such that for each $(x,m)\in \Omega$ we have $\chi(x,m,\cdot)\in L^{\infty}(\Omega)$ with the bound 
\begin{align}
\underset{b\geq 1}{\sup} \frac1b {\left\Vert \underset{(y,n)\in \Omega}{\sup}\left\vert \chi\right\vert (\cdot,y,n)\right\Vert_{L^{b}(\psi)} }< \Lambda , \label{Lb condition}    
\end{align}
for a fixed universal constant $\Lambda$ introduced in \cite{jabin2018quantitative}. Suppose that for any $I_{2k}\in \mathscr{R}_{N,2k}$ and any $J_{2k}\notin \mathcal{J}_{m_{I_{2k}},n_{I_{2k}}}$ there holds the cancellation rule   
\begin{align}
\int_{\Omega_{N}}\chi(x_{i_{1}},m_{i_{1}},x_{j_{1}},m_{j_{1}})\cdots\chi(x_{i_{2k}},m_{i_{2k}},x_{j_{2k}},m_{j_{2k}})\overline{\psi_{N}}(\mathbf{x}_{N},\mathbf{m}_{N})\ \de \mathbf{x}_{N}\de\mathbf{m}_{N}=0. \label{cancellation identity}   
\end{align}
Then, it holds that 
\begin{align}
\int_{\Omega_{N}} \exp\left(\left\vert \frac{1}{N}\sum_{i,j=1}^{N}\chi(x_{i},m_{i},x_{j},m_{j})\right\vert \right)\overline{\psi_{N}}(\mathbf{x}_{N},\mathbf{m}_{N}) \ \de \mathbf{x}_{N}\de\mathbf{m}_{N}\leq C \label{exp int}   
\end{align}
for some constant $C>0$. In particular, the integral in the left-hand side of \eqref{exp int} is well defined. 
\end{thm}
\begin{rem}
The original statement in \cite{jabin2018quantitative} concludes the bound 
\begin{align*}
\int_{\Omega_{N}} \exp\left( \frac{1}{N}\sum_{i,j=1}^{N}\chi(x_{i},m_{i},x_{j},m_{j}) \right)\overline{\psi_{N}}(\mathbf{x}_{N},\mathbf{m}_{N}) \ \de \mathbf{x}_{N}\de\mathbf{m}_{N}\leq C.    
\end{align*}
The bound with the absolute value, as stated above, follows immediately from the inequality $e^{\left\vert r\right\vert}\leq e^{r}+e^{-r}$. 
\end{rem}

In the following, we will prove that for choices of $\phi,\theta$ specified below, the assumptions \eqref{Lb condition}-\eqref{cancellation identity} of Theorem \ref{Theorem_jabin_cancel} are indeed fulfilled. 
First we verify that the required cancellations \eqref{cancellation identity} stated in the above Lemma are satisfied. 

\begin{lem}[Cancellation Lemma] \label{cancellation rule}
Let $\mathbf{a}$ and $S$ satisfy \textbf{H1} and suppose that $\psi\in \mathscr{P}(\Omega)\cap L^{1}(\Omega)$. Let 
\begin{align}
\phi(x,m,y,n)\coloneqq&\,\partial_{m}S(x,m,y,n)+S(x,m,y,n)\partial_{m}\log(\psi(x,m))
\notag\\&-\int_{\Omega}\partial_{m}S(x,m,y',n')\psi(y',n') \ \de y' \de n' \notag\\
&-\partial_{m}\log(\psi(x,m))\int_{\Omega}S(x,m,y',n')\psi(y',n') \ \de y'\de n' \label{def of phi}     
\end{align}  
and 
\begin{align}
\theta(x,m,y,n)=(m\mathbf{a}(y-x)-\mathbf{a}\star\mu[\psi](y))\nabla_{y} \log \psi(y,n). \label{def of teta}     
\end{align}
Then, for any $I_{2k}\in \mathscr{R}_{N,2k}$ and $J_{2k}\notin \mathcal{J}_{m_{I_{2k}},n_{I_{2k}}} $ it holds that 
\begin{align*}
\int_{\Omega_{N}}\phi(x_{i_{1}},m_{i_{1}},x_{j_{1}},m_{j_{1}})\dots \phi(x_{i_{2k}},m_{i_{2k}},x_{j_{2k}},m_{j_{2k}})\overline{\psi_{N}}(\mathbf{x}_{N},\mathbf{m}_{N})\ \de \mathbf{x}_{N}\de\mathbf{m}_{N}=0   
\end{align*}
and 
\begin{align*}
\int_{\Omega_{N}}\theta(x_{i_{1}},m_{i_{1}},x_{j_{1}},m_{j_{1}})\dots \theta(x_{i_{2k}},m_{i_{2k}},x_{j_{2k}},m_{j_{2k}})\overline{\psi_{N}}(\mathbf{x}_{N},\mathbf{m}_{N})\ \de \mathbf{x}_{N}\de\mathbf{m}_{N}=0.      
\end{align*}
\end{lem}
\begin{proof}
Define the effective set to be the set of all $I_{2k}$ with multiplicity different from $1$, i.e.  
\begin{align*}
\mathcal{E}_{N,2k}\coloneqq \left\{ I_{2k}\in\mathcal{T}_{N,2k}\left|\forall 1\leq l\leq N: a_{l}\neq1\right.\right\}.     
\end{align*}
There are three possible (potentially non-disjoint) cases for $I_{2k}\in \mathscr{R}_{N,2k}$ and $J_{2k}\notin \mathcal{J}_{m_{I_{2k}},n_{I_{2k}}} $: 
\begin{itemize}
    \item[i.]  $I_{2k}\notin \mathcal{E}_{N,2k}$, $I_{2k}\in \mathscr{R}_{N,2k}$ and there exist some $1\leq l \leq m_{I_{2k}}$ such that $b_{l}=0$.
    \item[ii.] $I_{2k} \notin \mathcal{E}_{N,2k}$, $I_{2k}\in \mathscr{R}_{N,2k}$ and there is some $l>m_{I_{2k}}+n_{I_{2k}}$ such that $b_{l}=1$. 
    \item[iii.] $I_{2k}\in \mathcal{E}_{N,2k}\cap \mathscr{R}_{N,2k}$.  
\end{itemize}
\textbf{Step 1. Cancellation for $\phi$}. \\
\textit{Case 1}. Suppose that i. holds. Since $I_{2k}\notin \mathcal{E}_{N,2k}$,  there exists some $1\leq r\leq N$ such that $a_{r}=1$ (note that $1\leq r\leq m_{I_{2k}}$ because $a_{r}\neq 1$ for all $r>m_{I_{2k}}$).  Therefore case i. implies that for all $1\leq r\leq m_{I_{2k}}$ (which is a non-empty inequality due to $m_{I_{2k}}\geq 1$) it holds $a_{r}=1$ and  there exist $1\leq l \leq m_{I_{2k}}$ such that $b_{l}=0$. 
\newline Taking $r=l$, due to the definitions of $a_r$ and $b_l$, this means that there exists a $\nu$ such that $i_{\nu}=l$ and $i_{\nu'}\neq r$ for all $\nu' \neq \nu$ and $j_{\nu'}\neq l$ for all $1\leq \nu'\leq 2k$. Therefore, by the Leibniz rule it holds that 
\begin{align}
\phi(x_{i_{\nu}},m_{i_{\nu}},x_{j_{\nu}},m_{j_{\nu}})\psi(x_{i_{\nu}},m_{i_{\nu}})=&\,\partial_{m_{i_{\nu}}}(S(x_{i_{\nu}},m_{i_{\nu}},x_{j_{\nu}},m_{j_{\nu}})\psi(x_{i_{\nu}},m_{i_{\nu}})) \notag \\
&-\int_{\Omega_{N}}\partial_{m_{i_{\nu}}}(S(x_{i_{\nu}},m_{i_{\nu}},y,n)\psi(x_{i_{\nu}},m_{i_{\nu}}))\psi(y,n)\ \de y \de n \notag\\
=&\,\partial_{m_{i_{\nu}}}\varphi(x_{i_{\nu}},m_{i_{\nu}},x_{j_{\nu}},m_{j_{\nu}}), \label{Leib rule}
\end{align}
where we have set 
\begin{align*}
\varphi(x,m,y,n)\coloneqq S(x,m,y,n)\psi(x,m)-\int_{\Omega}S(x,m,y',n')\psi(x,m)\psi(y',n')\ \de y'\de n'.      
\end{align*}
Recognizing that the product  $\prod_{\nu'\neq \nu}\phi(x_{i_{\nu'}},m_{i_{\nu'}},x_{j_{\nu'}},m_{j_{\nu'}})\prod_{l\neq i_{\nu}}\psi(x_{l},m_{l})$ is independent of $m_{i_{\nu}}$ we can invoke \eqref{Leib rule} and integrate by parts in order to find that  
 \begin{align*}
\int_{\Omega_{N}}&\phi(x_{i_{1}},m_{{i_{1}}},x_{j_{1}},m_{j_{{1}}})\cdots \phi(x_{i_{1}},m_{i_{i_{1}}},x_{j_{1}},m_{j_{{1}}}) \overline{\psi_{N}}(\mathbf{x}_{N},\mathbf{m}_{N})\ \de\mathbf{x}_{N}\de \mathbf{m}_{N}\\
&=\int_{\Omega_{N}}\partial_{m_{i_{\nu}}}\varphi(x_{i_{\nu}},m_{i_{\nu}},x_{j_{\nu}},m_{j_{\nu}})\prod_{\nu'\neq \nu}\phi(x_{i_{\nu'}},m_{i_{\nu'}},x_{j_{\nu'}},m_{j_{\nu'}})\prod_{l\neq i_{\nu}}\psi(x_{l},m_{l}) \ \de\mathbf{x}_{N}\de\mathbf{m}_{N}=0.   
 \end{align*}
\textit{Case 2}. Suppose that either ii. or iii. hold.  \\ 
If ii. holds, then we know that $m_{I_{2k}}\geq 1$ and there is some $l>m_{I_{2k}}+n_{I_{2k}}$ such that $b_{l}=1$.
%$I_{2k}$  
%is such that there exist some $1\leq r\leq m_{I_{2k}}$ such that $a_{r}=1$ and there is some $l>m_{I_{2k}}+n_{I_{2k}}$ such that $b_{l}=1$.
By definition of $b_{l}$ and the fact that $I_{2k}\in \mathscr{R}_{N,2k}$, this means that there is some $\nu$ such that $j_{\nu} = l$ and $j_{\nu'}\neq l$ for all $\nu' \neq \nu$ as well as $i_{\nu'}\neq l $ for all $1\leq \nu'\leq 2k$. \newline 
If iii. holds then $m_{I_{2k}}=0$ and therefore $J_{2k}\notin \mathcal{J}_{m_{I_{2k}},n_{I_{2k}}}$ implies that there exist some $l>n_{I_{2k}}$ such that $b_{l}=1$. So again, this means that there is $\nu $ such that $j_{\nu} = l$ such that $j_{\nu'}\neq l$ for all $\nu' \neq \nu$ and $i_{\nu'}\neq l $ for all $1\leq \nu'\leq 2k$. 
We then have 
\begin{align*}
\int_{\Omega_{N}}&\phi(x_{i_{1}},m_{i_{1}},x_{j_{1}},m_{j_{1}})\cdots \phi(x_{i_{2k}},m_{i_{2k}},x_{j_{2k}},m_{j_{2k}})\ \overline{\psi_{N}}(\mathbf{x}_{N},\mathbf{m}_{N})\ \de\mathbf{x}_{N}\de\mathbf{m}_{N}\\
=&\int_{\Omega_{N-1}}
\left(\int_{\Omega}\phi(x_{i_\nu},m_{i_{\nu}},x_{j_{\nu}},m_{j_{\nu}})\psi(x_{j_{\nu}},m_{j_{\nu}})\ \de x_{j_{\nu}}\de m_{j_{\nu}}\right)\\
&\times\prod_{l\neq j_\nu}\psi(x_{l},m_{l})\prod_{\nu'\neq\nu}\phi(x_{i_{\nu'}},m_{i_{\nu'}},x_{j_{\nu'}},m_{j_{\nu'}})\ \de \mathbf{x}'_{N-1}\de\mathbf{m}'_{N-1}.
\end{align*}
Here we apply the notation 
\begin{align*}
\de \mathbf{x}_{N-1}'\de\mathbf{m}_{N-1}'\coloneqq \de x_{j_{1}}\de m_{j_{1}}\cdots\hat{\de x_{j_{\nu}}}\hat{\de m_{j_{\nu}}}\cdots \de x_{N}\de m_{N}.     
\end{align*}
Note that, due to direct inspection of \eqref{def of phi}, for all $(x,m)\in \Omega$ we have 
\[
\int_{\Omega}\phi(x,m,y,n)\psi(y,n)\ \de y\de n=0.\
\]
 It follows that 
\[
\int_{\Omega_{N}}\phi(x_{i_{1}},m_{i_{1}},x_{j_{1}},m_{j_{1}})\cdots \phi(x_{i_{2k}},m_{i_{2k}},x_{j_{2k}},m_{j_{2k}})\overline{\psi_{N}}(\mathbf{x}_{N},\mathbf{m}_{N})\ \de\mathbf{x}_{N}\de\mathbf{m}_{N}=0.
\]
\textbf{Step 2. Cancellation for $\theta$}. As before we distinguish between 2 cases.\\ 
\textit{Case 1}. 
As discussed in Step 1, if i. holds there is some $i_{\nu}$ such that $i_{\nu'}\neq i_{\nu}$ for all $\nu'\neq \nu$ and $i_{\nu}\neq j_{\nu'}$ for all $\nu'$. 
In this case it holds that 
\begin{align*}
\int_{\Omega_{N}}&\theta(x_{i_{1}},m_{i_{1}},x_{j_{1}},m_{j_{1}})\cdots \theta(x_{i_{2k}},m_{i_{2k}},x_{j_{2k}},m_{j_{2k}})\overline{\psi_{N}}(\mathbf{x}_{N},\mathbf{m}_{N})\ \de\mathbf{x}_{N}\de\mathbf{m}_{N}\\
=&\int_{\Omega_{N-1}}\left(\int_{\Omega}\theta(x_{i_{\nu}},m_{i_{\nu}},x_{j_{\nu}},m_{j_{\nu}})\psi(x_{i_{\nu}},m_{i_{\nu}})\ \de x_{i_{\nu}}\de m_{i_{\nu}}\right)\\
&\times\prod_{l\neq i_{\nu}}\psi(x_{l},m_{l})\prod_{\nu'\neq \nu}\theta(x_{i_{\nu'}},m_{i_{\nu'}},x_{j_{\nu'}},m_{j_{\nu'}})\ \de \mathbf{x}_{N-1}'\de \mathbf{m}_{N-1}'=0,
\end{align*}
where we used that for each $(y,n)\in \Omega$
\begin{align*}
\int_{\Omega}\theta(x,m,y,n)\psi(x,m)\ \de x\de m=0.     
\end{align*}
\textit{Case 2}. As discussed in Step 1, if we assume that ii. or iii. hold, there is some $j_{\nu}$ such that $j_{\nu}\neq j_{\nu'}$ for all $\nu \neq \nu'$ and $i_{\nu'}\neq j_{\nu}$ for all $1\leq \nu'\leq 2k$. In this case, it holds that  
\begin{align*}
\int_{\Omega_{N}}&\theta(x_{i_{1}},x_{i_{1}},m_{j_{1}},n_{j_{1}})\cdots \theta(x_{i_{2k}},m_{i_{2k}},x_{j_{2k}},m_{j_{2k}})\overline{\psi_{N}}(\mathbf{x}_{N},\mathbf{m}_{N})\ \de \mathbf{x}_{N}\de \mathbf{m}_{N}\\
=&\int_{\Omega_{N-1}}\left(\int_{\Omega}\theta(x_{i_{\nu}},m_{i_{\nu}},x_{j_{\nu}},m_{j_{\nu}})\psi(x_{j_{\nu}},m_{j_{\nu}})\ \de x_{j_{\nu}}\de m_{j_{\nu}}\right)\\
&\times \prod_{l \neq j_{\nu}}\psi(x_{l},m_{l})\prod_{\nu' \neq \nu}\theta(x_{i_{\nu'}},m_{i_{\nu'}},x_{j_{\nu'}},m_{j_{\nu'}})\ \de \mathbf{x}_{N-1}'\de \mathbf{m}_{N-1}'.    \end{align*}
The inner integral vanishes because
\begin{align*}
\int_{\Omega}\theta(x,m,y,n)\psi(y,n)\ \de y\de n=\int_{\Omega}\nabla_{y} \psi(y,n)\left(m\mathbf{a}(x-y)-\mathbf{a}\star\mu[\psi](y)\right) \ \de y\de n=0   
\end{align*}
where we used integration by parts and that $\mathrm{div}(\mathbf{a})=0$. This finishes the proof. \end{proof}
\begin{rem}
In contrast to \cite{jabin2018quantitative}, in our case it only holds that $\int_{\Omega}\phi(x,m,y,n) \psi(y,n)\ \de y\de n=0$ but not necessarily $\int_{\Omega}\phi(x,m,y,n)\psi(x,m)\  \de x\de m=0$. In other words, $\phi$ satisfies a cancella\-tion rule only in one variable, unlike $\theta$ which satisfies a cancellation rule in both variables. Neverthe\-less, Lemma \ref{cancellation rule} shows that the cancellation for $I_{2k}\in \mathscr{R}_{N,2k}$ and $J_{2k}\notin \mathcal{J}_{m_{I_{2k}},n_{I_{2k}}}$ is still valid, which is all we need in order to apply Theorem \ref{Theorem_jabin_cancel}.  
\end{rem}
Next we move to check that the condition \eqref{Lb condition} is satisfied for $\phi,\theta$ defined in Lemma \ref{cancellation rule}. 
\begin{lem}
Let the assumptions of Lemma \ref{cancellation rule} hold and let   $\phi$ and $\theta$ be given by \eqref{def of phi} and \eqref{def of teta} respectively. Assume also that $\mu[\psi]\in L^{1}(\mathbb{T}^{d})$.  Then, it holds that 
\begin{align*}
\sup_{b\geq 1}\frac1b {\left\Vert \underset{(y,n)\in \Omega}{\sup}\left\vert \phi\right\vert(\cdot,y,n) \right\Vert_{L^{b}(\psi)}}\leq \gamma_{1} \quad
\mbox{ and }
\quad
\underset{b\geq 1}{\sup} \frac1b {\left\Vert \underset{(y,n)\in \Omega}{\sup}\left\vert \theta\right\vert(\cdot,y,n)\right\Vert_{L^{b}(\psi)} }\leq \gamma_{2}   
\end{align*}
with $\gamma_{1}=2\overline{S}_{1}+2\overline{S}_{0}\left\Vert \partial_{m}\log(\psi)\right\Vert_{\infty}$ and 
\begin{align*}
\gamma_{2}=\left\Vert \nabla_{x} \log\psi\right\Vert_{\infty}\left\Vert \mathbf{a}\right\Vert_{\infty}\left(\left\Vert \mu[\psi]\right\Vert_{1}+\underset{b\geq 1}{\sup}\frac{ \mathfrak{M}^{\frac{1}{b}}_{b,1}(\psi)}{b}\right),       
\end{align*}
where $\overline{S}_{0}, \overline{S}_{1}, \mathfrak{M}_{b,1}(\psi)$ are defined in (\textbf{H1}) and (\textbf{H2}) respectively.
\label{est on phi teta} 
\end{lem}
\begin{proof}
Using \eqref{def of phi}, by splitting into its four summands, we write $\phi=\sum_{k=1}^{4}\phi_{k}$ for simplicity. Moreover, we observe that 
\begin{align*}
\left(\int_{\Omega}\underset{(y,n)\in \Omega}{\sup}\left\vert \partial_{m}S\right\vert^{b}(x,m,y,n)\psi(x,m)\ \de x \de m\right)^{\frac{1}{b}}\leq \left\Vert \partial_{m}S\right\Vert_{\infty}\leq \overline{S}_{1}    \end{align*}
and 
\begin{align*}
\int_{\Omega}\partial_{m}S(x,m,y',n')\psi(y',n') \ \de y' \de n'\leq \left\Vert \partial_{m}S \right\Vert_{\infty}\leq \overline{S}_{1},     
\end{align*}
so that for $k=1,3$ we have  
\begin{align}
\left\Vert \underset{(y,n)\in \Omega}{\sup}\phi_{k}(\cdot,y,n)\right\Vert_{L^{b}(\psi)}\leq \left\Vert \partial_{m}S\right\Vert_{\infty}\leq \overline{S}_{1}.\label{phi1phi3}   \end{align}
In addition, it holds  
\begin{align}
&\left(\int_{\Omega}\underset{(y,n)\in \Omega}{\sup}\left\vert \partial_{m}\log \psi\right\vert^{b} (x,m)\left\vert S \right\vert^{b} (x,m,y,n) \psi(x,m)\ \de x\de m\right)^{\frac{1}{b}} \notag\\
&\leq \overline{S}_{0} \left(\int_{\Omega}\left\vert \partial_{m}\log \psi\right\vert^{b}(x,m)\psi(x,m) \ \de x \de m \right)^{\frac{1}{b}} \notag
\leq \overline{S}_{0}\left\Vert \partial_{m}\log\psi\right\Vert_{\infty}.        
\end{align}
Hence, for $k=2,4$ we have 
\begin{align}
\sup_{b\geq 1}\frac1b {\left\Vert \underset{(y,n)\in \Omega}{\sup}\left\vert\phi_{k}\right\vert(\cdot,y,n) \right\Vert_{L^{b}(\psi)}}\leq \overline{S}_{0}\left\Vert \partial_{m}\log\psi\right\Vert_{\infty}. \label{phi2phi4}  
\end{align}
The combination of \eqref{phi1phi3} and \eqref{phi2phi4} yields the first announced inequality. For the second statement let us write $\theta(x,m,y,n)=\theta_{1}+\theta_{2}$
by splitting \eqref{def of teta} into its two components. 
This yields
\begin{align*}
\int_{\Omega}\underset{(y,n)\in \Omega}{\sup}\left\vert \theta_{1}(x,m,y,n)\right\vert^{b}\psi(x,m)  \ \de x\de m&\leq \left\Vert \nabla_{x} \log \psi\right\Vert_{\infty}^{b}\left\Vert \mathbf{a}\right\Vert_{\infty}^{b}  \int_{\Omega}m^{b}\psi(x,m)\ \de x\de m\\
&\leq \left\Vert \nabla_{x} \log \psi\right\Vert_{\infty}^{b}\left\Vert \mathbf{a}\right\Vert_{\infty}^{b} \mathfrak{M}_{b,1}(\psi),  \end{align*}
and  
\begin{align*}
\int_{\Omega}\underset{(y,n)\in \Omega}{\sup}\left\vert \theta_{2}\right\vert^{b}(x,m,y,n)\psi(x,m)\ \de x\de m&\leq 
\left\Vert \mathbf{a}\star \mu[\psi]\right\Vert_{\infty}^{b}\left\Vert \nabla_{x} \log\psi\right\Vert_{\infty}^{b}\\
&\leq
\left\Vert \mathbf{a}\right\Vert_{\infty}^{b}\left\Vert \mu[\psi]\right\Vert_{1}^{b}\left\Vert \nabla_{x} \log\psi\right\Vert_{\infty}^{b}.   
\end{align*}
Therefore, we conclude that 
\begin{align*}
\underset{b\geq 1}{\sup}\frac1b {\left\Vert 
\underset{(y,n)\in \Omega}{\sup}\left\vert \theta\right\vert(\cdot,y,n) \right\Vert_{L^{b}(\psi)} }\leq \left\Vert \nabla_{x} \log\psi\right\Vert_{\infty}\left\Vert \mathbf{a}\right\Vert_{\infty}\left(\left\Vert \mu[\psi]\right\Vert_{1}+\underset{b\geq 1}{\sup}\frac{ \mathfrak{M}^{\frac{1}{b}}_{b,1}(\psi)}{b}\right). 
\end{align*}
\end{proof}
Combining the estimates of this section as well as Theorem \ref{Theorem_jabin_cancel} leads to the final result of this section.\\

\textit{Proof of Theorem \ref{final exp reminder}}. 
Note that 
by the Cauchy-Schwarz inequality we have 
\begin{align}
&\int_{\Omega_{N}}\exp\left(\left\vert \mathcal{R}_{N}+\mathcal{\mathcal{S}}_{N}\right\vert \right)\overline{\psi_{N}}\ \dd \mathbf{x}_{N}\dd \mathbf{m}_{N}\leq \int_{\Omega_{N}}\exp(\left\vert \mathcal{R}_{N}\right\vert 
)\exp\left(\left\vert \mathcal{S}_{N}\right\vert \right)\overline{\psi_{N}}\ \dd \mathbf{x}_{N}\dd \mathbf{m}_{N} \notag\\
&\leq \left(\int_{\Omega_{N}}\exp(2\left\vert \mathcal{R}_{N}\right\vert)\overline{\psi_{N}} \ \dd\mathbf{x}_{N}\dd\mathbf{m}_{N}\right)^{\frac{1}{2}}\left(\int_{\Omega_{N}}\exp(2\left\vert \mathcal{S}_{N}\right\vert)\overline{\psi_{N}} \ \dd\mathbf{x}_{N}\dd\mathbf{m}_{N}\right)^{\frac{1}{2}}. \label{cauchy schwarz}   \end{align}
To bound the integrals in \eqref{cauchy schwarz} we apply Theorem \ref{Theorem_jabin_cancel} separately with $\chi = 2\phi  $ and $\chi=2\theta$ as defined in \eqref{def of phi}, \eqref{def of teta} and $\psi(t,x,m)$ a solution to \eqref{mean field PDE compact}. 
In order to fulfill condition \eqref{Lb condition} with a given $\Lambda>0$, we need to be able to take $\gamma_1$ and $\gamma_2$ in Lemma \ref{est on phi teta} small enough, and we now explain why is this possible. 

In view of Lemmas \ref{Positivity}-\ref{log grad est weight}, we need to choose $\overline{S}_{0},\overline{S}_{1},T_{\ast\ast }>0$ such that 
\begin{align}
4\overline{S}_{1}+4\overline{S}_{0}K_{\ref{ine for partialmlogpsi}}(\mathcal{C},\overline{S}_{0},\overline{S}_{1},\overline{S}_{2},T_{\ast\ast })<\frac \Lambda 2 \label{1st require}   \end{align}
where we recall that  
\begin{align*}
K_{\ref{ine for partialmlogpsi}}(\mathcal{C},\overline{S}_{0},\overline{S}_{1},\overline{S}_{2},T_{\ast\ast})\coloneqq \mathcal{C}^{2}\left(1+\overline{S}_{2}T_{\ast\ast}e^{(\overline{S}_{0}+\overline{S}_{1})T_{\ast\ast}}\right)e^{(\overline{S}_{0}+2\overline{S}_{1})T_{\ast\ast}}.     
\end{align*}
This would ensure that 
\begin{align*}
\sup_{b\geq 1}\frac1b {\left\Vert \underset{(y,n)\in \Omega}{\sup}\left\vert \phi\right\vert(\cdot,y,n) \right\Vert_{L^{b}(\psi)}}\leq \gamma_{1}=\frac \Lambda 2
\end{align*}
holds for all $\phi$ defined in \eqref{def of phi} and where $\psi = \psi(t, \cdot)$ is taken to be the solution of \eqref{mean field PDE compact} for any time $0\leq t\leq T_{\ast\ast}$. 
Analogously, in view of Lemmas \ref{lemma truncated}-\ref{logarthimic gradient estimate}, we need to choose $\overline{S}_{0},\overline{S}_{1},\left\Vert \mathbf{a}\right\Vert_{\infty}$ such that
\begin{align}
2\left\Vert \mathbf{a}\right\Vert_{\infty}\mathcal{C}K_{\ref{grad x linfty est}}(T_{\ast\ast})e^{(\overline{S}_{0}+\overline{S}_{1})T_{\ast\ast }}\left(\overline{\mu}+\overline{S}_{0}T_{\ast\ast}+e^{2(\overline{S}_{0}+\overline{S}_{1})T_{\ast\ast}}\mathfrak{M}_{\mathrm{in}}\right)<\frac \Lambda 2. \label{2nd require}  \end{align}
This would ensure that 
\begin{align*}
\sup_{b\geq 1}\frac1b {\left\Vert \underset{(y,n)\in \Omega}{\sup}\left\vert \theta\right\vert(\cdot,y,n) \right\Vert_{L^{b}(\psi)}}\leq \gamma_{2}=\frac \Lambda 2
\end{align*}
holds for all $\theta$ defined in \eqref{def of teta} where $\psi = \psi(t, \cdot)$ is taken to be the solution of \eqref{mean field PDE compact} for any time $0\leq t\leq T_{\ast \ast}$.
Clearly, conditions \eqref{2nd require}-\eqref{2nd require} can be met for $\left\Vert \mathbf{a}\right\Vert_{\infty},\overline{S}_{0},\overline{S}_{1},T_{\ast\ast}$ small enough.  
\qed

\section*{Acknowledgments}
The research of JAC and AH is partially supported by the Advanced Grant Nonlocal\--CPD (Non\-local PDEs for Complex Particle Dynamics: Phase Transitions, Patterns and Synchro\-ni\-zation) of the European Research Council Executive Agency (ERC) under the European Union's Horizon 2020 research and innovation programme (grant agreement No. 883363).
IBP and JAC were also partially supported by EPSRC grant number EP/V051121/1.
JAC was also partially supported by the "Maria de Maeztu" Excellence Unit IMAG, reference CEX2020-001105-M, funded by MCIN/AEI/10.13039/501100011033/

\bibliographystyle{plain}
\bibliography{references}

@article{piccoli2019wasserstein,
  title={A Wasserstein norm for signed measures, with application to nonlocal transport equation with source term},
  author={Piccoli, Benedetto and Rossi, Francesco and Tournus, Magali},
  journal={arXiv preprint arXiv:1910.05105},
  year={2019}
}

@article{pouradier2021mean,
  title={Mean-field limit of collective dynamics with time-varying weights},
  author={Pouradier Duteil, Nastassia},
  journal={arXiv e-prints},
  pages={arXiv--2103},
  year={2021}
}

@article{mcquade2019social,
  title={Social dynamics models with time-varying influence},
  author={McQuade, Sean and Piccoli, Benedetto and Pouradier Duteil, Nastassia},
  journal={Mathematical Models and Methods in Applied Sciences},
  volume={29},
  number={04},
  pages={681--716},
  year={2019},
  publisher={World Scientific}
}

@article{ben2024graph,
  title={The graph limit for a pairwise competition model},
  author={Ben-Porat, Immanuel and Carrillo, Jos{\'e} A and Jabin, Pierre-Emmanuel},
  journal={Journal of Differential Equations},
  volume={413},
  pages={329--369},
  year={2024},
  publisher={Elsevier}
}

@article{porat2023mean,
  title={Mean field limit for 1d opinion dynamics with Poission interaction and time dependent weights},
  author={Ben-Porat, Immanuel and Carillo, JA and Galtung, S},
  journal={arXiv preprint arXiv:2306.01099},
  year={2023}
}

@article{jabin2016mean,
  title={Mean field limit and propagation of chaos for Vlasov systems with bounded forces},
  author={Jabin, Pierre-Emmanuel and Wang, Zhenfu},
  journal={Journal of Functional Analysis},
  volume={271},
  number={12},
  pages={3588--3627},
  year={2016},
  publisher={Elsevier}
}

@article{jabin2018quantitative,
  title={Quantitative estimates of propagation of chaos for stochastic systems with {$W^{-1,\infty}$} kernels},
  author={Jabin, Pierre-Emmanuel and Wang, Zhenfu},
  journal={Inventiones mathematicae},
  volume={214},
  pages={523--591},
  year={2018},
  publisher={Springer}
}

@book{villani2009optimal,
  title={Optimal transport: old and new},
  author={Villani, C{\'e}dric},
  volume={338},
  year={2009},
  publisher={Springer}
}

@article{piccoli2014generalized,
  title={Generalized Wasserstein distance and its application to transport equations with source},
  author={Piccoli, Benedetto and Rossi, Francesco},
  journal={Archive for Rational Mechanics and Analysis},
  volume={211},
  pages={335--358},
  year={2014},
  publisher={Springer}
}

@article{ayi2024large,
  title={Large-population limits of non-exchangeable particle systems},
  author={Ayi, Nathalie and Duteil, Nastassia Pouradier},
  journal={arXiv preprint arXiv:2401.07748},
  year={2024}
}

@article{paul2022microscopic,
  title={From microscopic to macroscopic scale equations: mean field, hydrodynamic and graph limits},
  author={Paul, Thierry and Tr{\'e}lat, Emmanuel},
  journal={arXiv preprint arXiv:2209.08832},
  year={2022}
}

@book{dembo2009large,
  title={Large deviations techniques and applications},
  author={Dembo, Amir},
  year={2009},
  publisher={Springer}
}

@article{ben2025singular,
  title={Singular flows with time-varying weights},
  author={Ben-Porat, Immanuel and Carrillo, Jos{\'e} A and Jabin, Pierre-Emmanuel},
  journal={arXiv preprint arXiv:2503.02276},
  year={2025}
}

@article{biccari2019dynamics,
  title={Dynamics and control for multi-agent networked systems: A finite-difference approach},
  author={Biccari, Umberto and Ko, Dongnam and Zuazua, Enrique},
  journal={Mathematical Models and Methods in Applied Sciences},
  volume={29},
  number={04},
  pages={755--790},
  year={2019},
  publisher={World Scientific}
}

@article{ayi2023graph,
  title={Graph limit for interacting particle systems on weighted random graphs},
  author={Ayi, Nathalie and Duteil, Nastassia Pouradier},
  journal={arXiv preprint arXiv:2307.12801},
  year={2023}
}

@article{gkogkas2025mean,
  title={Mean field limits of co-evolutionary signed heterogeneous networks},
  author={Gkogkas, Marios Antonios and Kuehn, Christian and Xu, Chuang},
  journal={European Journal of Applied Mathematics},
  pages={1--44},
  year={2025},
  publisher={Cambridge University Press}
}

@article{jabin2025mean,
  title={Mean-field limit of non-exchangeable systems},
  author={Jabin, Pierre-Emmanuel and Poyato, David and Soler, Juan},
  journal={Communications on Pure and Applied Mathematics},
  volume={78},
  number={4},
  pages={651--741},
  year={2025},
  publisher={Wiley Online Library}
}

@article{bresch2023mean,
  title={Mean field limit and quantitative estimates with singular attractive kernels},
  author={Bresch, Didier and Jabin, Pierre-Emmanuel and Wang, Zhenfu},
  journal={Duke Mathematical Journal},
  volume={172},
  number={13},
  pages={2591--2641},
  year={2023},
  publisher={Duke University Press}
}

@article{bresch2019modulated,
  title={Modulated free energy and mean field limit},
  author={Bresch, Didier and Jabin, Pierre-Emmanuel and Wang, Zhenfu},
  journal={S{\'e}minaire Laurent Schwartz - EDP applications},
  pages={1--22},
  year={2019}
}

@article{berner2023adaptive,
  title={Adaptive dynamical networks},
  author={Berner, Rico and Gross, Thilo and Kuehn, Christian and Kurths, J{\"u}rgen and Yanchuk, Serhiy},
  journal={Physics Reports},
  volume={1031},
  pages={1--59},
  year={2023},
  publisher={Elsevier}
}

@article{duchet2023mean,
  title={Mean-field approximations with adaptive coupling for networks with spike-timing-dependent plasticity},
  author={Duchet, Benoit and Bick, Christian and Byrne, {\'A}ine},
  journal={Neural computation},
  volume={35},
  number={9},
  pages={1481--1528},
  year={2023},
  publisher={MIT Press One Rogers Street, Cambridge, MA 02142-1209, USA journals-info~…}
}

@article{ayi2026mean,
  title={Mean-field limits for interacting particles on general adaptive dynamical networks},
  author={Ayi, Nathalie},
  journal={arXiv preprint arXiv:2601.03742},
  year={2026}
}

@article{carrillo2011global,
  title={Global-in-time weak measure solutions and finite-time aggregation for nonlocal interaction equations},
  author={Carrillo, Jos{\'e} A and DiFrancesco, Marco and Figalli, Alessio and Laurent, Thomas and Slep{\v{c}}ev, Dejan},
  year={2011}
}

\end{document}